\documentclass{amsart}
\usepackage{amsmath, amscd, amssymb, amsthm}
\usepackage{bbm}
\usepackage{latexsym}
\usepackage{amsfonts}
\usepackage{graphicx}
\usepackage[all,cmtip]{xy}
\usepackage[colorlinks,linkcolor=blue,breaklinks=blue,urlcolor=blue,citecolor=blue,anchorcolor=blue,pagebackref]{hyperref}%
\usepackage{geometry}
\newtheorem{theorem}{Theorem}
\newtheorem{lemma}{Lemma}

\newtheorem{remark}{Remark}

\newtheorem{conjecture}{Conjecture}

\renewcommand*\backref[1]{}
\renewcommand*\backrefalt[4]{ \ifcase #1 \or (cited on page #2) \else (cited on pages #2) \fi}

\newcommand{\be}{\begin{equation}}
\newcommand{\ee}{\end{equation}}
\newcommand{\bea}{\begin{eqnarray}}
\newcommand{\eea}{\end{eqnarray}}

\newcommand{\vs}{\vspace{0.5cm}}

\def\XXint#1#2#3{{\setbox0=\hbox{$#1{#2#3}{\int}$ }
\vcenter{\hbox{$#2#3$ }}\kern-.6\wd0}}

\begin{document}

\title[On Hermitian manifolds with constant mixed curvature]{On Hermitian manifolds with constant mixed curvature}

\author{Shuwen Chen}
\address{Shuwen Chen. School of Mathematical Sciences, Chongqing Normal University, Chongqing 401331, China}
\email{{3153017458@qq.com}}\thanks{Chen is supported by Chongqing graduate student research grant No. CYB240227.  The corresponding author Zheng is partially supported by National Natural Science Foundations of China with the grant No.12141101 and 12471039,  Chongqing Normal University grant 24XLB026, and is supported by the 111 Project D21024.}

\author{Fangyang Zheng}
\address{Fangyang Zheng. School of Mathematical Sciences, Chongqing Normal University, Chongqing 401331, China}
\email{20190045@cqnu.edu.cn; franciszheng@yahoo.com} \thanks{}

\subjclass[2020]{53C55 (primary), 53C05 (secondary)}
\keywords{Hermitian manifold, Chern connection, holomorphic sectional curvature,  mixed curvature, Hermitian Lie algebras}

\begin{abstract}
In a recent work, Kai Tang conjectured that any compact Hermitian manifold with non-zero constant mixed curvature must be K\"ahler. He confirmed the conjecture in complex dimension $2$ and for Chern K\"ahler-like manifolds in general dimensions. In this paper, we  verify his conjecture for several special types of Hermitian manifolds, including complex nilmanifolds, solvmanifolds with complex commutators, almost abelian Lie groups, and Lie algebras containing a $J$-invariant abelian ideal of codimension $2$. We also verify the conjecture for all compact balanced threefolds when the Bismut connection has  parallel torsion. These results provide partial evidence towards the validity of Tang's conjecture.
\end{abstract}

\maketitle

\tableofcontents

\section{Introduction}\label{intro}

The simplest kind of Riemannian manifolds are the {\em space forms,} which are complete Riemannian manifolds with constant sectional curvature. They are quotients of the sphere \( S^n \), the Euclidean space \( \mathbb{R}^n \), or the hyperbolic space \( \mathbb{H}^n \),  equipped with (a scaling of) the standard metric. In the complex case, the sectional curvature of a Hermitian metric can no longer be constant unless the metric is flat. So instead one requires that the holomorphic sectional curvature be constant. A complete K\"ahler manifold with constant holomorphic sectional curvature is called a {\em complex space form}. Their universal covers are either the complex projective space \( \mathbb{C} \mathbb{P}^n \), or the complex Euclidean space \( \mathbb{C}^n \), or the  complex hyperbolic space \( \mathbb{C} \mathbb{H}^n \), equipped with (a scaling of) the standard metric.

Recall that for a given  Hermitian manifold $(M^{n},g)$, the {\em holomorphic sectional curvature} $H$ of the Chern connection is defined by
\begin{align}
H(X)=R(X,\overline{X},X,\overline{X})/|X|^{4} \nonumber \,\,,
\end{align}
where $X$ is any  non-zero type $(1, 0)$ tangent vector, and $R$ is the curvature tensor of the Chern connection. When the metric $g$ is non-K\"ahler, $H$ cannot determine the entire $R$ in general. So it is natural to wonder when can $H$ be constant. The following is a long-standing conjecture in Hermitian geometry:
\begin{conjecture}[{\bf Constant Holomorphic Sectional Curvature Conjecture}] \label{conj1}
Let $(M^{n},g)$ be a compact Hermitian manifold with $H$ equal to a constant $c$. If $c\neq0$, then $g$ is K\"{a}hler and if $c=0$, then $g$ is Chern flat (namely $R=0$).
\end{conjecture}

First of all, if one replaces the Chern connection by another metric connection, such as Levi-Civita, Bismut, or Gauduchon connections in the conjecture, then one gets related conjectures/questions. See \cite{CN, ChenZ, ChenZ1} for more discussions on this. Secondly, note that Boothby \cite{Boothby} classified compact Chern flat manifolds in 1958 as the set of all compact quotients of complex Lie groups, equipped with left-invariant metrics compatible with the complex structure. Thirdly, the compactness assumption in the conjecture is necessary, and there are counterexamples in the noncompact case.

When $n=2$, the conjecture is known to be true. The $c \leq 0$ case  was established by Balas-Gauduchon \cite{BG} in 1985, while the more difficult $c > 0$ case was solved by Apostolov-Davidov-Muskarov \cite{ADM} in 1996.

For $n\geq 3$ the conjecture remains an open problem, with only a few partial results known so far. For instance, in \cite{DGM}, Davidov-Grantcharov-Muskarov showed that the only twistor space with constant holomorphic sectional curvature is the complex space form \( \mathbb{C} \mathbb{P}^3 \). In \cite{Tang}, K.\,Tang confirmed the conjecture under the additional assumption that $g$ is Chern K\"ahler-like, meaning that $R$ obeys all K\"ahler symmetries. In \cite{CCN}, Chen-Chen-Nie proved the conjecture under the additional assumption that $g$ is locally conformally K\"ahler and $c\leq 0$. In \cite{ZhouZ}, Zhou-Zheng proved that any compact balanced threefold with zero {\em real bisectional curvature}, a notion introduced in \cite{XYangZ} which is slightly stronger than $H$, must be Chern flat. In \cite{LZ} and \cite{RZ}, Li-Zheng or Rao-Zheng confirmed the conjecture for complex nilmanifolds or compact Hermitian manifolds that are {\em Bismut K\"ahler-like,} meaning that the curvature $R^b$ of the Bismut connection $\nabla^b$ obeys all K\"ahler symmetries.

The notion of \emph{mixed curvature} $\mathcal{C}_{\alpha,\beta}$ was introduced by Chu-Lee-Tam \cite{CLT} as a convex combination of the (first) Ricci curvature and holomorphic sectional curvature of the Chern connection:
\begin{equation*}
\mathcal{C}_{\alpha,\beta}(X) = \alpha \operatorname{Ric}(X,\overline{X}) + \beta H(X),
\end{equation*}
where  $\alpha, \beta \in \mathbb{R}$, $X \in T^{1,0}M$ is a type $(1,0)$ tangent vector of unit length, and $\operatorname{Ric}$ is the first Ricci curvature of the Chern connection, defined by $\operatorname{Ric}(X,\overline{X}) = \sum_{i=1}^n R(X,\overline{X}, e_i, \overline{e}_i)$ where $e$ is any local unitary frame. Throughout  this paper, we will always assume that $(\alpha, \beta) \neq (0,0)$. 

In the above notation, we have $\mathcal{C}_{1,0}=\operatorname{Ric}$ and $\mathcal{C}_{0,1}=H$. When the metric $g$ is K\"{a}hler, the mixed curvature encompasses several interesting curvature conditions. For instance, $\mathcal{C}_{1,1}$ corresponds to the $Ric^{+}(X, \overline{X})$ introduced by Ni \cite{N2}, and $\mathcal{C}_{1,-1}$ corresponds to the {\em orthogonal Ricci curvature} $Ric^{\perp}(X, \overline{X})$ introduced by Ni-Zheng \cite{NZ1} (see also \cite{BT} for some generalizations). Additionally, if $k$ is an integer between $1$ and $n$, then $\mathcal{C}_{k-1,n-k}$ is closely related to the $k$-Ricci curvature $Ric_{k}$ introduced by Ni \cite{N1}.

For compact K\"{a}hler manifolds with $\mathcal{C}_{\alpha,\beta}\geq 0$ or $\mathcal{C}_{\alpha,\beta}\leq 0$, there have been important progress made in recent years. For example, Yang \cite{Yang2018} proved that a compact K\"ahler manifold with $\mathcal{C}_{0,1} > 0$ must be projective and rationally connected, confirming a conjecture by Yau \cite{Yau}. Ni \cite{N2} generalized this result, showing that the same holds if $Ric_k > 0$ for some $1 \leq k \leq n$. Matsumura \cite{Mat1,Mat2} established structure theorems for projective manifolds with $\mathcal{C}_{0,1} \geq 0$, while Chu-Lee-Tam \cite{CLT} and others proved that compact K\"ahler manifolds with $\mathcal{C}_{\alpha,\beta} > 0$, where $\alpha > 0$ and $3\alpha + 2\beta \geq 0$, must be projective and simply-connected. Zhang-Zhang \cite{ZZ} extended Yang's result to the quasi-positive case, confirming a conjecture by Yang. Tang \cite{Tang2} demonstrated the projectivity of compact K\"ahler manifolds with quasi-positive $\mathcal{C}_{\alpha,\beta}$. More recently, Chu-Lee-Zhu \cite{CLZ} established a structure theorem for compact K\"ahler manifolds with $\mathcal{C}_{\alpha,\beta} \geq 0$. In the non-positive case, Wu-Yau \cite{WuYau} confirmed a conjecture of Yau that projective K\"ahler manifolds with $\mathcal{C}_{0,1} < 0$ must have ample canonical line bundle. Tosatti-Yang \cite{TosattiYang} showed that the projectivity assumption can be removed in Wu-Yau's Theorem, and Chu-Lee-Tam \cite{CLT} proved that compact K\"ahler manifolds with $\mathcal{C}_{\alpha,\beta} < 0$ and $\alpha \geq 0$, $\beta \geq 0$, must have ample canonical bundle.

These recent developments on mixed curvature are mainly focused on K\"ahler manifolds, and much less is known about the non-K\"ahler case. In \cite{Tang25}, Tang proposed the following conjecture:

\begin{conjecture}[\bf{Tang}]\label{conj2} Let $(M^{n}, g)$ be a compact Hermitian manifold with $\mathcal{C}_{\alpha,\beta}= c$, where $c$ is a constant. If  $c\neq0$, then $g$ is K\"ahler. \end{conjecture}

Note that in the conjecture one may assume that $\beta \neq 0$. The $\beta =0$ case corresponds to (first) Chern Ricci curvature, where the metric is automatically K\"ahler when ${\mathcal C}_{1,0}=c \neq 0$, as in this case the K\"ahler form is a constant multiple of the  Chern Ricci form, which is always closed. 

For the $c=0$ case,  the vanishing of the mixed curvature $\mathcal{C}_{\alpha,\beta}$ in general  does not imply $R = 0$, as Tang have already demonstrated in \cite{Tang25} on isosceles Hopf manifolds.

In complex dimension $n=2$, Tang \cite{Tang25} proved that any compact Hermitian surface with constant mixed curvature must be either K\"ahler or an isosceles Hopf surface. For $n \geq 3$, Tang also obtained results for locally conformally K\"ahler manifolds and for Chern K\"ahler-like manifolds.

Recall that a \emph{Lie-Hermitian manifold} is a compact quotient $M = G/\Gamma$ of a Lie group $G$ by a discrete subgroup $\Gamma \subseteq G$, where the complex structure and metric on $G$ are left-invariant. If $G$ is nilpotent or solvable, $M$ is called a \emph{complex nilmanifold} or \emph{complex solvmanifold}, respectively. Lie-Hermitian manifolds form a large and distinctive class of locally homogeneous Hermitian manifolds. They are often used as testing ground  for conjectures in Hermitian geometry.

Denote by ${\mathfrak g}$ the Lie algebra of $G$. Left-invariant metrics on $G$ correspond to metrics (namely, inner products) on ${\mathfrak g}$, and left-invariant complex structures on $G$ correspond to complex structures on ${\mathfrak g}$, defined as linear transformations $J: {\mathfrak g} \to {\mathfrak g}$ satisfying $J^2 = -I$ and the integrability condition
\begin{equation} \label{integrability}
[x,y] - [Jx,Jy] + J[Jx,y] + J[x,Jy] =0, \ \ \ \ \ \forall \ x,y \in {\mathfrak g}.
\end{equation}
A Lie algebra ${\mathfrak g}$ equipped with a complex structure $J$ and a compatible metric $g = \langle \cdot, \cdot \rangle$ is called a \emph{Hermitian Lie algebra}, and it will be denoted by $({\mathfrak g}, J, g)$. From now on, we will use the same letters to denote the metric (or complex structure) on $G$ and ${\mathfrak g}$.

As a partial evidence, we will confirm Conjecture \ref{conj2} for Lie-Hermitian manifolds in four special cases: complex nilmanifolds,  {\em almost abelian Lie algebras} (namely, those containing an abelian ideal of codimension $1$), Lie algebras with $J$-invariant abelian ideal of codimension $2$, and solvable Lie algebras with complex commutators. The last case means that the Lie algebra ${\mathfrak g}$ of $G$ is solvable and satisfies $J{\mathfrak g}' = {\mathfrak g}'$, where ${\mathfrak g}' = [{\mathfrak g}, {\mathfrak g}]$ is the commutator.

\begin{theorem} \label{thm} Let $(M^n, g)$ be a Lie-Hermitian manifold with universal cover $(G, J, g)$. Assume  that the Lie algebra ${\mathfrak g}$ of $G$ satisfies one of the following: 
\begin{itemize}
\item [(i)]  ${\mathfrak g}$ is nilpotent, or
\item [(ii)]  ${\mathfrak g}$ is solvable with $J{\mathfrak g}' = {\mathfrak g}'$, where ${\mathfrak g}' = [{\mathfrak g}, {\mathfrak g}]$, or
\item [(iii)] $\mathfrak g$ is unimodular and almost abelian,  or
\item [(iv)]  $\mathfrak g$ is unimodular and contains an abelian ideal ${\mathfrak a}$ of codimension 2 with $J{\mathfrak a} = {\mathfrak a}$. 
\end{itemize}
If $\mathcal{C}_{\alpha,\beta} = c$, then $c=0$. If $\mathcal{C}_{\alpha,\beta} =0$ and $\beta \neq 0$, then $g$ must be Chern flat. 
\end{theorem}

In particular, Conjecture \ref{conj2} holds for each of these four types of Lie-Hermitian manifolds. When $\beta =0$, the mixed curvature is the first Chern Ricci curvature, and in each of the above four cases the Chern Ricci curvature could vanish identically while not being Chern flat. So the assumption $\beta \neq 0$ is necessary in the $c=0$ case of Theorem \ref{thm}. 

When $\alpha =0$,  namely for the holomorphic sectional curvature case $\mathcal{C}_{0,1}=H$, Theorem \ref{thm} is known to be true: (i) is due to Li-Zheng \cite{LZ}, (ii) is due to Huang-Zheng \cite{HuangZ}, while (iii) and (iv) are proven by Li-Zheng in \cite{LZ25}, where they also discussed the constant holomorphic sectional curvature problem for the Levi-Civita connection. 

\begin{remark}
The referee kindly pointed out the following highly relevant facts to us, to which we are deeply grateful.

(1). For any complex nilmanifold, its first Chern Ricci curvature vanishes identically by the result of Lauret and Valencia \cite[Proppsition 2.1]{LV}. Therefore $\mathcal{C}_{\alpha,\beta} = \beta H$ for such a manifold, and case (i) of Theorem \ref{thm} holds by \cite{LZ}.

(2). For any almost abelian Lie algebra ${\mathfrak g}$, let ${\mathfrak a}\subset {\mathfrak g}$ be an abelian ideal of codimension $1$, then ${\mathfrak a}_J= {\mathfrak a} \cap J{\mathfrak a}$ is always an ideal of ${\mathfrak g}$ by \cite[Lemma 4.1]{AL}. Therefore, ${\mathfrak g}$ contains a $J$-invariant abelian ideal of codimension $2$, so case (iii) is included in case (iv). 

(3). The argument in the proof of case (ii) can be extended to cover all unimodular Hermitian Lie algebras $({\mathfrak g},J,g)$ such that ${\mathfrak g}'+J{\mathfrak g}'$ is nilpotent. In particular, if ${\mathfrak g}$ is solvable and its nilradical is $J$-invariant, then Conjecture \ref{conj2} holds.
\end{remark}

Based on his suggestions, we will skip the proof of case (i) to avoid redundancy. We will keep the proof of case (iii), however, for readers' convenience, as the proof of case (iv) is somewhat complicated yet analogous to that of case (iii), with the latter being more transparent. For item (3) above, we will add the following statement, but the credit should really go to the anonymous referee.

\begin{theorem} \label{thm1b}
Let $(M^n, g)$ be a Lie-Hermitian manifold with universal cover $(G, J, g)$. Assume  that the Lie algebra ${\mathfrak g}$ of $G$ satisfies the condition that ${\mathfrak g}' + J {\mathfrak g}'$ is nilpotent, where ${\mathfrak g}' = [{\mathfrak g}, {\mathfrak g}]$.  
If $\mathcal{C}_{\alpha,\beta} = c$, then $c=0$. If $\mathcal{C}_{\alpha,\beta} =0$ and $\beta \neq 0$, then $g$ must be Chern flat. 
\end{theorem}

Next let us recall the Bismut connection $\nabla^b$ (also known as the Strominger connection in some literature) of a Hermitian manifold $(M^n,g)$, which is the unique connection satisfying the conditions that $\nabla^b g = 0$, $\nabla^b J = 0$, and its torsion is totally skew-symmetric (\cite{Bismut,Strominger}). The metric $g$ is called Bismut torsion-parallel (BTP for brevity)  if $\nabla^bT^b=0$, where $T^b$ is the torsion of the Bismut connection. BTP manifolds form a relatively large and interesting class of special Hermitian manifolds. It contains all Bismut K\"ahler-like (abbreviated as BKL) manifolds, meaning when the curvature of $\nabla^b$ obeys all K\"ahler symmetries. See \cite{ZZ-Crelle} for a proof of the AOUV Conjecture \cite{AOUV} by Angella-Otal-Ugarte-Villacampa, which can be alternatively stated as  BKL $\subset$ BTP.  Clearly, all Bismut flat manifolds (\cite{WYZ}) are  BKL, thus are examples of BTP manifolds. Other examples of BTP manifolds are Vaisman manifolds, namely,  locally conformally K\"ahler manifolds whose Lee form is parallel under the Levi-Civita connection.

In complex dimension $2$, BTP = BKL = Vaisman. Such surfaces were fully classified by Belgun \cite{Belgun} in $2000$. However, for $n \geq 3$, BKL and Vaisman are disjoint, and their union is a proper subset of the set of all non-balanced BTP manifolds. Additionally, for $n \geq 3$, there exist balanced (non-K\"ahler) BTP manifolds. Balanced BTP manifolds form a highly restrictive and interesting set, containing examples that are Chern flat or Fano. For further discussions on BTP manifolds, we refer readers to the preprint \cite{ZhaoZ24, ZhaoZ25} and the references therein. Now let us examine Conjecture \ref{conj2} for BTP  manifolds. We start with the balanced ones. In this case we are only able to confirm the conjecture in complex dimension $3$:

\begin{theorem} \label{thm2}
Let $(M^3,g)$ be a compact balanced BTP threefold which is not K\"ahler. If  $\mathcal{C}_{\alpha,\beta} = c$, then $c$ must be zero. If  $\mathcal{C}_{\alpha,\beta} = 0$ and $(\alpha ,\beta )\neq (0,0)$, then either $(M^3,g)$ is Chern flat, or $\beta =0$ and $(M^3,g)$ is Chern Ricci flat of middle type. 
\end{theorem}

Note that compact non-K\"ahler balanced BTP threefolds are either compact quotients of $SL(2,{\mathbb C})$ (the Chern flat case), or the Wallach threefold (the Fano case), or in the so-called {\em middle type}, which means the rank of the {\em Streets-Tian tensor} $B$  (\cite[Formula (4)]{ST11}, where it was denoted as $Q^2$) is $2$. The local geometry of the middle type ones are determined by two real-valued smooth functions $x$, $y$, and it is Chern Ricci flat (or equivalently Bismut Ricci flat) if and only if $x=y=0$ identically. 

For non-balanced BTP manifolds, we have the following:

\begin{theorem} \label{thm3}
Let $(M^n,g)$ be a compact non-balanced BTP manifold with $\mathcal{C}_{\alpha,\beta} = c$. Then $c$ must be zero. 
\end{theorem}

In other words Conjecture \ref{conj2} holds for all  compact non-balanced BTP manifolds. Note that such a manifold  could never be Chern flat, yet there are such examples where $\mathcal{C}_{\alpha,\beta} = 0$ identically for some special values of $(\alpha ,\beta)$. As a wild speculation, we propose the following:

\begin{conjecture} \label{conj3}
Let $(M^n,g)$ be a compact non-balanced BTP manifold. If $\mathcal{C}_{\alpha,\beta} = 0$ identically, then either $\beta =0$, or $k\alpha +\beta =0$ for some integer $k$ between $2$ and $n$. 
\end{conjecture}

We remark that it is necessary to exclude the above `exceptional values' for $(\alpha , \beta)$. The $\beta =0$ case corresponds to the (first) Chern Ricci curvature. Take any complex nilmanifold that is BTP (there are  both balanced and non-balanced BTP nilmanifolds), it is always Chern Ricci flat by Lauret-Valencia \cite{LV}, namely, ${\mathcal C}_{1,0}=0$ identically. Also, if we consider the standard Hopf manifold 
$$X^k=({\mathbb C}^k\setminus \{ 0\})/\langle f\rangle , \ \ \ f(z)=2z, \ \ \ \ \mbox{with metric} \ \ \ \omega = \sqrt{-1} \,\frac{\partial \overline{\partial}\, |z|^2}{|z|^2}. $$
Then it is easy to see that $\operatorname{Ric}=kH$, hence ${\mathcal C}_{1,-k}\equiv 0$  on $X^k$. By taking $M^n=X^k\times {\mathbb T}^{n-k}$ to be the product of $X^k$ with a flat complex torus, we see that on the non-balanced BTP manifold $M^n$ we have ${\mathcal C}_{1,-k}\equiv 0$. Here $k$ could be any integer between $2$ and $n$.

The article is organized as follows. In the next section, we will set up the notations and collect some known results from existing literature that will be used in later proofs. In the sections 3, 4, and 5,  we will give
proofs to the cases (ii), (iii), and (iv) of Theorem \ref{thm}, respectively. In the last section, we will give a
proof to Theorems \ref{thm2} and \ref{thm3}.

\vspace{0.3cm}

\section{Preliminaries}
Let $(M^n,g)$ be a Hermitian manifold and denote by  $\omega$ the  K\"ahler form associated with $g$. Denote by $\nabla$, $\nabla^b$ the Chern and Bismut connection, respectively.

Fix any $p \in M^n$, let $\{ e_1, \ldots , e_n\} $ be a frame of $(1,0)$-tangent vectors of $M^n$ in a neighborhood of $p$, and let $\{ \varphi_1, \ldots , \varphi_n\}$ be the dual coframe of $(1,0)$-forms. We will also write  $e=\,^t\!(e_1, \ldots , e_n)$ and $\varphi = \,^t\!( \varphi_1, \ldots , \varphi_n)$ and view them as column vectors. Write $g=\langle \cdot  , \cdot \rangle $ and extend it bi-linearly over ${\mathbb C}$. We will follow the notations of  \cite{VYZ,YZ18Cur,YZ18Gau}.

For the Chern connection $\nabla$, let us denote by $\theta$,  $\Theta$ respectively the matrices of connection and
curvature under the frame $e$, and by $\tau$ the column vector of the
torsion $2$-forms under $e$. Then the structure
equations and Bianchi identities are given by
\begin{equation*}
\left\{ \begin{array}{llll} d \varphi = - \ ^t\!\theta \wedge \varphi + \tau,   \\
d  \theta = \theta \wedge \theta + \Theta,  \\
d \tau = - \ ^t\!\theta \wedge \tau + \ ^t\!\Theta \wedge \varphi, \\
d  \Theta = \theta \wedge \Theta - \Theta \wedge \theta. \end{array} \right.
\end{equation*}
The entries of $\Theta$ are all $(1,1)$-forms, while the entries of the column vector $\tau $ are all $(2,0)$-forms, under any frame $e$.
Similar symbols such as $\theta^b,\Theta^b$ and $\tau^b$ are applied to Bismut connection, then we have
\begin{equation*}
\Theta^b = d\theta^b -\theta^b \wedge \theta^b.
\end{equation*}

Let $\gamma =\nabla^b-\nabla$ be the tensor, and for simplicity we also write $\gamma = \theta^b-\theta$ for its matrix representation under $e$. It follows from \cite{YZ18Cur} that when $e$ is unitary, the matrices $\gamma$ is given by
\begin{equation*}
\gamma_{ij} = \sum_k \{ T^j_{ik}\varphi_k  - \overline{T^i_{jk}} \,\overline{\varphi}_k \} , \ \
\end{equation*}
where $T^k_{ij}$, satisfying $T^k_{ji}=-T^k_{ij}$, are the components of the Chern torsion, given by
\[ \tau_k =\frac{1}{2}  \sum_{i,j=1}^n T_{ij}^k \,\varphi_i\wedge \varphi_j \ = \sum_{1\leq i<j\leq n}  \ T_{ij}^k \varphi_i\wedge \varphi_j.\]
Let $R_{i\bar{j}k\bar{\ell}} $ be the short hand notation for $R(e_i, \overline{e}_j, e_k, \overline{e}_{\ell})$, where $R$ denotes the curvature tensor of the Chern connection $\nabla$. Similarly, denote by $R^b$ the curvature of $\nabla^b$. We have
$$ \Theta_{ij} = \sum_{k,\ell =1}^n R_{k\overline{\ell} i \overline{j}} \, \varphi_k \wedge \overline{\varphi}_{\ell}, \ \ \ \ \Theta^b_{ij} = \sum_{k,\ell =1}^n \big( \frac{1}{2} R^b_{k\ell i \overline{j}} \,\varphi_k \wedge \varphi_{\ell}  + \frac{1}{2} R^b_{\bar{k}\bar{\ell} i \overline{j}} \,\overline{\varphi}_k \wedge \overline{\varphi}_{\ell} + R^b_{k\overline{\ell} i \overline{j}} \,\varphi_k \wedge \overline{\varphi}_{\ell} \big) . $$
As in \cite{LZ}, let us introduce the {\em symmetrization} of a $(4,0)$-tensor:
$$ \widehat{R}_{i\bar{j}k\bar{\ell}} = \frac{1}{4} \big(  R_{i\bar{j}k\bar{\ell}} + R_{k\bar{j}i\bar{\ell}} + R_{i\bar{\ell}k\bar{j}} + R_{k\bar{\ell}i\bar{j}} \big) , $$
for any $1\leq i,j,k,\ell \leq n$. Then it is easy to see that

\begin{lemma} \label{lemma1}
Let $(M^n,g)$ be a Hermitian manifold. Then $\, \mathcal{C}_{\alpha,\beta}=c \, \Longleftrightarrow$
\begin{equation} \label{mix1} 
 4\beta\widehat{R}_{i\bar{j}k\bar{\ell}} +
\alpha (R_{i\bar{j}}g_{k\bar{l}}+R_{k\bar{j}}g_{i\bar{l}}+
R_{i\bar{l}}g_{k\bar{j}}+R_{k\bar{l}}g_{i\bar{j}})
= 2c\big( g_{i\bar{j}} g_{k\bar{\ell}} + g_{i\bar{\ell }} g_{k\bar{j}}  \big), \ \ \ \forall \ 1\leq i,j,k,\ell \leq n,
\end{equation}
where $R_{i\bar{j}}=\operatorname{Ric}(e_i, \overline{e}_j)$ denotes the components of the first Chern Ricci curvature. In particular,  when the frame $e$ is unitary, we have $R_{i\bar{j}}=\sum_{s=1}^n R_{i\bar{j}s\bar{s}}$ and 
\begin{eqnarray}
&& \beta R_{i\bar{i}i\bar{i}}+\alpha R_{i\bar{i}} \ =\ c, \ \ \ \ \forall \ 1\leq i \leq n,\label{chun1} \\
&& 4\beta \widehat {R}_{i\bar{i}k\bar{k}}+\alpha (R_{i\bar{i}}+R_{k\bar{k}}) \ = \ 2c, \ \ \ \ \forall \ 1\leq i\neq k \leq n. \label{chun2}
\end{eqnarray}
\end{lemma}

Here and from now on we will use the abbreviation:
\begin{equation} \label{eq:vandw}
w=\sum_r T^r_{ik} \overline{T^r_{j\ell }}, \ \  v^j_i= \sum_r T^j_{ir} \overline{T^k_{\ell r}}, \ \  v^{\ell}_i= \sum_r T^{\ell}_{ir} \overline{T^k_{j r}}, \ \  v^j_k= \sum_r T^j_{kr} \overline{T^i_{\ell r}}, \ \  v^{\ell}_k= \sum_r T^{\ell}_{kr} \overline{T^i_{j r}}.
\end{equation}
Following \cite{ChenZ2}, we can take the $(1,1)$-part in $\Theta^b=d\theta^b - \theta^b  \wedge \theta^b $ and obtain
\begin{lemma} \label{lemma2}
Let $(M^n,g)$ be a Hermitian manifold. Under any local unitary frame $e$ it holds that
$$ R^b_{i\bar{j}k\bar{\ell}} - R_{i\bar{j}k\bar{\ell}} =  T_{ik;\overline{j}}^{\ell} + \overline{T^k_{j\ell ; \overline{i}}}   + v^{\ell}_i - v^j_i -v^{\ell}_k - w,  \ \ \ \
\ \widehat{R}_{i\bar{j}k\bar{\ell}}= \ \widehat{R}^b_{i\bar{j}k\bar{\ell}} + \widehat{v}.\ \ \ \forall \ 1\leq i,j,k,\ell \leq n,$$
for any $1\leq i,j,k,\ell \leq n$, where $w$ and $v^j_i$ etc. are given by (\ref{eq:vandw}), and indices after the semicolon stand for covariant derivatives with respect to the Bismut connection $\nabla^b$, while
$$ 4\hat{v} \ = \ v^j_i + v^{\ell}_k + v^{\ell}_i + v^j_k. $$
\end{lemma}
For our later proofs, which rely on the classification result for balanced BTP threefolds in \cite{ZhaoZ25}, we will also need the following result from \cite{ChenZ2}:
\begin{lemma}
Let $(M^n,g)$ be a BTP manifold, then under any local unitary frame  $e$ it holds that
\begin{equation} \label{eq:Rbhat}
 \widehat{ R}^b_{i\bar{j}k\bar{\ell}}  = R^b_{i\bar{j}k\bar{\ell}} + \frac{1}{2}(w + v^j_i +  v^{\ell}_k - v^{\ell}_i - v^j_k),\ \ \ \ \ \forall \ 1\leq i,j,k,\ell \leq n.
\end{equation}
\end{lemma}

Next, we consider Lie-Hermitian manifolds, which are compact quotients of Lie groups by their discrete subgroups, where both the complex structures and metrics are left-invariant. In the past  few decades, the Hermitian geometry of Lie-Hermitian manifolds have been extensively studied from various aspects by many people, including A. Gray, S. Salamon, L. Ugarte, A. Fino, L. Vezzoni, F. Podest\`a, D. Angella, A. Andrada, Lauret, and others. There is a large amount of literature on this topic, and here we mention only a small sample: \cite{AU}, \cite{CFGU},  \cite{FP3}, \cite{FinoTomassini09},  \cite{GiustiPodesta}, \cite{Salamon}, \cite{Ugarte}, \cite{WYZ}, \cite{ZZ-Crelle}, \cite{ZZ-JGP}. For broader discussions on non-K\"ahler Hermitian geometry, see works such as \cite{AI}, \cite{AT}, \cite{Fu}, \cite{STW}, \cite{Tosatti} and the references therein.

Let $M=G/\Gamma$ be a compact quotient of a Lie group $G$ by a discrete subgroup $\Gamma \subseteq G$, equipped with (the descend of) a left-invariant complex structure $J$ on $G$.  Let ${\mathfrak g}$ be the Lie algebra of $G$, then $J$ corresponds to a complex structure (which for convenience we will still denote by $J$) on  ${\mathfrak g}$. Similarly, any left-invariant metric $g$ on $G$ compatible with $J$ will correspond to an inner product $g=\langle \cdot , \cdot \rangle$ on ${\mathfrak g}$ such that $\langle Jx,Jy\rangle = \langle x,y\rangle$, for any $x,y\in {\mathfrak g }$. Here again for convenience we use the same letter to denote the corresponding metric on the Lie algebra.

As in \cite{GuoZ2}, let ${\mathfrak g}^{\mathbb C}$ be the complexification of ${\mathfrak g}$, and write ${\mathfrak g}^{1,0}= \{ x-\sqrt{-1}Jx \mid x \in {\mathfrak g}\} \subseteq {\mathfrak g}^{\mathbb C}$. The integrability condition (\ref{integrability}) means that ${\mathfrak g}^{1,0}$ is a complex Lie subalgebra of ${\mathfrak g}^{\mathbb C}$. Extend $g=\langle \cdot , \cdot \rangle $ bi-linearly over ${\mathbb C}$, and let $e=\{ e_1, \ldots , e_n\}$ be a basis of ${\mathfrak g}^{1,0}$, which will be called a {\em frame} of $({\mathfrak g},J)$ or $({\mathfrak g}, J,g)$.  Denote by $\varphi$ the coframe dual to $e$, namely, a basis of the dual vector space $({\mathfrak g}^{1,0})^{\ast}$ such that $\varphi_i(e_j)=\delta_{ij}$, $\forall$ $1\leq i,j\leq n$. We will write
\begin{equation} \label{CandD}
C^j_{ik} = \varphi_j( [e_i,e_k] ), \ \ \ \ \ \  D^j_{ik} = \overline{\varphi}_i( [\overline{e}_j, e_k] )
\end{equation}
for the structure constants, which is  equivalent to
\begin{equation} \label{CandD2}
[e_i,e_j] = \sum_k C^k_{ij}e_k, \ \ \ \ \ [e_i, \overline{e}_j] = \sum_k \big( \overline{D^i_{kj}} e_k - D^j_{ki} \overline{e}_k \big) .
\end{equation}
In dual terms, the above is also equivalent to the (first) structure equation:
\begin{equation} \label{eq:structure}
d\varphi_i = -\frac{1}{2} \sum_{j,k} C^i_{jk} \,\varphi_j\wedge \varphi_k - \sum_{j,k} \overline{D^j_{ik}} \,\varphi_j \wedge \overline{\varphi}_k, \ \ \ \ \ \ \forall \  1\leq i\leq n.
\end{equation}
Differentiate the above, we get the  first Bianchi identity, which is equivalent to the Jacobi identity in this case:
\begin{equation*}
\left\{  \begin{split}  \sum_r \big( C^r_{ij}C^{\ell}_{rk} + C^r_{jk}C^{\ell}_{ri} + C^r_{ki}C^{\ell}_{rj} \big) \ = \ 0,  \hspace{3.2cm}\\
 \sum_r \big( C^r_{ik}D^{\ell}_{jr} + D^r_{ji}D^{\ell}_{rk} - D^r_{jk}D^{\ell}_{ri} \big) \ = \ 0, \hspace{3cm} \\
 \sum_r \big( C^r_{ik}\overline{D^r_{j \ell}}  - C^j_{rk}\overline{D^i_{r \ell}} + C^j_{ri}\overline{D^k_{r \ell}} -  D^{\ell}_{ri}\overline{D^k_{j r}} +  D^{\ell}_{rk}\overline{D^i_{jr}}  \big) \ = \ 0,  \end{split} \right.
\end{equation*}
for any $1\leq i,j,k,\ell\leq n$. When $G$ has a compact quotient, it (or equivalently, its Lie algebra ${\mathfrak g}$) is necessarily unimodular, that is, $\mbox{tr}(ad_x)=0$ for any $x\in {\mathfrak g}$. In terms of structure constants, we have
\begin{equation*}
{\mathfrak g} \ \, \mbox{is unimodular}  \ \ \Longleftrightarrow  \ \ \sum_r \big( C^r_{ri} + D^r_{ri}\big) =0 , \, \ \forall \ i.
\end{equation*}
Note that so far we did not assume $e$ to be unitary. When $e$ is unitary, namely, when $\langle e_i, \overline{e}_j\rangle = \delta_{ij}$ for any $1\leq i,j\leq n$, then the formula for Chern connection form, the Chern torsion components $T^j_{ik}$ and Chern curvature components $R_{i\bar{j}k\bar{\ell}}$  under $e$ are particularly simple (\cite{GuoZ2, LZ}):
\begin{eqnarray}
 &&  \theta_{ij} \ = \ \sum_{k=1}^n \big( D^j_{ik}\varphi_k  -\overline{D^i_{jk}} \overline{\varphi_k} \big), \ \ \ \ 
  T^j_{ik}  =   - C^j_{ik} -   D^{j}_{i k} + D^{j}_{k i} , \ \ \ \forall \ 1\leq i,j,k \leq n, \label{torsion} \\
 && R_{i\bar{j}k\bar{\ell}}\ = \   \sum_{s=1}^n \big( D^s_{ki}\overline{D^s_{\ell j}} - D^{\ell}_{si}\overline{D^k_{s j}} - D^j_{si}\overline{D^k_{ \ell s}} - \overline{D^i_{sj}} D^{\ell}_{k s} \big), \  \ \ \forall \ 1\leq i,j,k,\ell \leq n, \label{curvature}
\end{eqnarray}
In particular, we have
\begin{eqnarray}
&&  R_{i\overline{i}i\overline{i}} \ = \  \sum_{r=1}^n \big( |D^r_{ii}|^2 - |D^i_{ri}|^2 - 2 {\mathfrak R}\mbox{e} \{  D^i_{ri} \overline{D^i_{ir}} \} \big), \ \ \ \forall \ 1\leq i \leq n, \label{HD} \\
&& R_{i\overline{i}s\overline{s}} \ = \ \sum_{r=1}^n \big( |D^r_{si}|^2 - |D^s_{ri}|^2 - 2 {\mathfrak R}\mbox{e} \{  D^i_{ri} \overline{D^s_{sr}} \} \big), \ \ \ \forall \ 1\leq i, s \leq n, \label{chun3} \\
&& 
\widehat{R}_{i\overline{i}k\overline{k}} \ =  \ \frac{1}{4}\sum_{r=1}^n \big( |D^r_{ki}+D^r_{ik}|^2 - |D^k_{ri}|^2 -|D^i_{rk}|^2 - 2 {\mathfrak R}\mbox{e} \{ D^k_{rk} \overline{D^i_{ri} } +   D^i_{ri} \overline{D^k_{kr} }  + \label{RhatD} \\
&& \hspace{1.4cm} + \,D^k_{rk} \overline{D^i_{ir} } + D^i_{rk} \overline{D^i_{kr} } + D^k_{ri} \overline{D^k_{ir} }\} \big), \ \ \ \ \ \forall \  1\leq i,k \leq n. \nonumber
\end{eqnarray}

\vspace{0.3cm}

\section{Solvable Lie algebras with complex commutators}

In this section, we will prove case (ii) of Theorem \ref{thm}. Let $G$ be a Lie group equipped with a left-invariant metric $g=\langle \cdot , \cdot \rangle $ and a compatible left-invariant complex structure $J$. First let us recall the famous result of Salamon \cite[Theorem 1.3]{Salamon} for nilpotent groups: 

\begin{theorem} [Salamon] Let $G$ be a nilpotent Lie group of dimension $2n$ equipped with a left invariant complex structure. Then there exists a coframe $\varphi =\{ \varphi_1, \ldots , \varphi_n\}$ of left invariant $(1,0)$-forms on $G$ such that $$ d\varphi_1 =0, \ \ \ d\varphi_i = {\mathcal I} \{\varphi_1, \ldots , \varphi_{i-1}\} , \ \ \ \forall \ 2\leq i\leq n, $$
where ${\mathcal I}$ stands for the ideal in  exterior algebra of the complexified cotangent bundle generated by those $(1,0)$-forms.
\end{theorem}

Clearly, one can choose the above coframe $\varphi$ so that it is also unitary. In terms of the structure constants $C$ and $D$ given by (\ref{CandD}), this means
\begin{equation}
C^j_{ik}=0  \ \ \ \mbox{unless} \ \ j>i \ \mbox{or} \ j>k; \ \ \ \ \ D^j_{ik}=0  \ \ \ \mbox{unless} \ \ i>j.  \label{Salamon}
\end{equation}
Let $e$ be a unitary frame dual to $\varphi$ satisfying (\ref{Salamon}). We will call $e$ a {\em Salamon frame} from now on. Under such a frame one always has
\begin{equation}
D^n_{ik} = 0, \ \ \ \forall \  1\leq i,k\leq n.  \label{Duppern}
\end{equation}

Now suppose that $G$ is a unimodular Lie group equipped with a left-invariant complex structure $J$ and a compatible left-invariant metric $g$. Denote by ${\mathfrak g}$ its Lie algebra and ${\mathfrak g}'=[{\mathfrak g},{\mathfrak g}]$ its commutator, also denote by $J$, $g=\langle \cdot , \cdot \rangle$ the corresponding complex structure and metric on ${\mathfrak g}$. 

{\em We assume that ${\mathfrak g}' +  J{\mathfrak g}'$ is nilpotent.}

Note that  when ${\mathfrak g}$ is nilpotent (case (i) of Theorem \ref{thm}), or when ${\mathfrak g}$ is solvable with $J{\mathfrak g}'={\mathfrak g}'$ (case (ii) of Theorem \ref{thm}), or when ${\mathfrak g}$ is solvable and $J{\mathfrak n}={\mathfrak n}$ where ${\mathfrak n}$ is the nilradical of ${\mathfrak g}$, then ${\mathfrak g}' +  J{\mathfrak g}'$ is automatically nilpotent, as it is well-known that the commutator of a solvable Lie algebra is always nilpotent. Under the above assumption,  ${\mathfrak g}'+ J {\mathfrak g}'$ is a nilpotent Lie algebra. Denote its (real) dimension by $2r$, and the (real) dimension of ${\mathfrak g}$ by $2n$. 

A unitary frame $e$ of ${\mathfrak g}$ is called {\em admissible} if $\{ e_1, \ldots , e_r\}$ is a Salamon frame of ${\mathfrak g}'+  J{\mathfrak g}'$, equipped with the restriction complex structure and metric. In particular,   ${\mathfrak g}'+J{\mathfrak g}'$ is spanned by $\{ e_i+\overline{e}_i, \sqrt{-1}(e_i - \overline{e}_i)\}_{1\leq i\leq r}$. Throughout this section, we will always make the following convention on the range of indices:
$$ 1\leq i,j,\cdots \leq r, \ \ \ \ \ \ r\!+\!1\leq \gamma, \eta , \cdots \leq n, \ \ \ \ \ \ 1\leq a,b,\cdots \leq n. $$
From (\ref{CandD2}) and the definition of ${\mathfrak g}'$, we obtain \begin{equation}  \label{eq:res1}
C^{\gamma}_{ab}=D^{a}_{\gamma b}=0, \ \ \ \ \forall \ r\!+\!1\leq \gamma \leq n, \ \ \forall \ 1\leq a,b\leq n.
\end{equation}
Since ${ e_1, \ldots, e_r }$ is a Salamon frame for ${\mathfrak g}'+  J{\mathfrak g}'$, we have
\begin{equation}  \label{eq:res2}
C^j_{ik}=0  \ \  \mbox{unless} \ j>i \ \mbox{or} \ j>k; \ \ \ \ D^j_{ik}=0 \  \ \mbox{unless} \ i>j, \ \ \ \forall \ 1\leq i,j,k\leq r.
\end{equation}
This will be a key property used repeatedly. The readers are referred to \cite{LZ} for a more detailed discussion on this. 

\begin{lemma} \label{lemmacd=0}
Let $({\mathfrak g}, J,g)$ be a unimodular Lie algebra equipped with a Hermitian structure such that ${\mathfrak g}'+J{\mathfrak g}'$ is nilpotent and the Chern connection of $g$ has constant mixed curvature: $\mathcal{C}_{\alpha, \beta}=c$ where $\beta \neq 0$. Then $c=0$ and $D^{\gamma}_{j\eta}=D^j_{ik}=D^{\gamma}_{ij}=0$, for any $1\leq i,j,k\leq r,\,r\!+\!1\leq \gamma , \eta \leq n$.
\end{lemma}

\begin{proof}
Since $D^*_{\gamma *}=0$, by (\ref{HD}), (\ref{chun3}) and (\ref{eq:res2}) we have $\sum_{t=1}^n R_{\gamma\bar \gamma t \bar t}=0,$ $\sum_{t=1}^n R_{i\bar i t \bar t}=0$, and
$$R_{i\bar i i\bar i}=\sum_{s=1}^n  |D^s_{ii}|^2 - \sum_{s=1}^r |D^i_{si}|^2, \ \ \ \ \ R_{\gamma\bar \gamma \gamma \bar \gamma}=- \sum_{s=1}^r |D^\gamma_{s\gamma}|^2,$$
for any $1\leq i\leq r, \ r+1\leq \gamma \leq n.$ Thus, by (\ref{chun1}) we have
$$ \beta R_{i\bar i i\bar i}=\beta \big(\sum_{s=1}^n  |D^s_{ii}|^2 - \sum_{s=1}^r |D^i_{si}|^2 \big) = c,$$
for each $1\leq i \leq r$. In particular,
$$c=\beta R_{r\overline{r}r\overline{r}} = \beta \sum_{s=1}^r |D^s_{rr}|^2, \ \ \ \ \  \ c= \beta R_{\gamma\overline{\gamma}\gamma\overline{\gamma}} =- \beta \sum_{s=1}^r |D^\gamma_{s\gamma}|^2, \ \ \ \ \forall \ r+1\leq \gamma \leq n.  $$
Therefore we deduce $c=0$, $ D^*_{r r}=0$, and
\begin{eqnarray}
&& D^\gamma_{i \gamma} \ = \  0, \ \ \ \ \ \forall \ 1\leq i\leq r, \ r+1\leq \gamma \leq n,\label{eq:alpha} \ \ \\
&& \sum_{s=1}^n  |D^s_{ii}|^2 \ = \  \sum_{s=1}^r |D^i_{si}|^2, \ \ \ \ \  \forall \ 1\leq i \leq r. \label{eq:i}
\end{eqnarray}
Note that if we make any unitary change on $\{ e_{r\!+\!1},\ldots , e_n\}$, then formula (\ref{eq:res1}) and (\ref{eq:res2}) would not be affected and $e$ would remain to be an admissible frame. So formula (\ref{eq:alpha}) actually gives us $D^{X}_{jX}=0$, for any $j$ and any $X=\sum_{\gamma =r\!+\!1}^n X_{\gamma}e_{\gamma}$. That is, $\sum_{\gamma ,\eta=r\!+\!1}^n \overline{X}_{\gamma} X_{\eta} D^{\gamma}_{j\eta }=0$. This leads to $D^{\gamma}_{j\eta}=0$, for any $j$ and any $\gamma$, $\eta$.

We have $\sum_{t=1}^n R_{i\bar i t \bar t}=0$, $\forall \ 1\leq i\leq r$. Fix any  $1\leq i<k\leq r$. By (\ref{chun2}) and $\beta \neq 0$ we obtain $ \widehat{R}_{i\overline{i}k \overline{k}}= 0$. Hence by (\ref{RhatD}) we get
\begin{equation}
 0 =  \sum_{s=1}^n |D^s_{ki}+D^s_{ik}|^2 -  \sum_{s=1}^r \big( |D^k_{si}|^2 + |D^i_{sk}|^2 + 2{\mathfrak R}\mbox{e} \{ D^i_{si}\overline{D^k_{sk}} + \overline{D^i_{sk}} D^i_{ks} \} \big). \label{eq:iikk}
\end{equation}
When deriving the above equality, we used the fact that $D^{\ast}_{\gamma \ast}=0$ for any $\gamma >r$ by (\ref{eq:res1}). Let $k=r$ in (\ref{eq:iikk}), and by (\ref{eq:res2}) we get
\begin{equation}  \label{eq:iirr}
 \sum_{s=1}^n |D^s_{ri}+D^s_{ir}|^2 \, = \, \sum_{s=1}^r \big(  |D^i_{sr}|^2 + 2{\mathfrak R}\mbox{e} \{  \overline{D^i_{sr}} D^i_{rs} \} \big) , \ \ \ \ \ \ \forall \ 1\leq i\leq r.
 \end{equation}
Now by the same argument as in the proof of Theorem $1$ in Huang-Zheng \cite{HuangZ} (see also the proof of Theorem E in \cite{LZ}), we end up with
\begin{equation}  \label{eq:goal}
D^{\gamma}_{j\eta}=D^j_{ik}=D^{\gamma}_{ij}=0,  \ \ \ \ \ \ \ \ \forall\ 1\leq i,j,k\leq r, \ \ \ \forall \ r\!+\!1\leq \gamma , \eta \leq n.
 \end{equation}
 This completes the proof of the lemma.
\end{proof}

Now we are ready to finish the proof of case (ii) of Theorem \ref{thm}, which is a special case of Theorem \ref{thm1b}.

\begin{proof}[{\bf Proof of Theorem \ref{thm1b}}]
Let $({\mathfrak g}, J,g)$ be a unimodular Lie algebra equipped with a Hermitian structure. Assume that ${\mathfrak g}'+J{\mathfrak g}'$ is nilpotent and the Chern connection of $g$ has constant mixed curvature: $\mathcal{C}_{\alpha, \beta}=c$, where $\beta \neq 0$. Then $c=0$ by Lemma \ref{lemmacd=0}. Our goal is  to show that the Chern curvature of $g$ must vanish identically: $R=0$. Let $e$ be an admissible frame of ${\mathfrak g}$ so that (\ref{eq:res1}) and (\ref{eq:res2}) hold. By (\ref{eq:goal}) the only possibly non-zero $D$ components are $ D^j_{i\gamma}$, where $ 1\leq i,j\leq r$ and $r\!+\!1 \leq \gamma \leq n$. In particular, $D^{\gamma}_{\ast\ast} = D^{\ast}_{\gamma \ast}=0$. So by (\ref{curvature}) we know that
\begin{equation} \label{eq:alphaat34}
 R_{\ast \bar{\ast} \gamma \bar{\ast}} =  R_{\ast \bar{\ast} \ast \bar{\gamma} }=0, \ \ \ \ \ \ \forall \ r\!+\!1 \leq \gamma \leq n.
 \end{equation}
Therefore $R_{a\bar{b}c\bar{d}}=0$ if at least three of the indices belong to $\{ r\!+\!1, \ldots , n\}$. Also by (\ref{curvature}) we have
$$ R_{a\bar{b}c\bar{d}} = \sum_{s=1}^n D^s_{ca} \overline{D^s_{db} } - \sum_{s=1}^n D^d_{sa} \overline{D^c_{sb} }, $$
since for the last two terms in the right hand side of (\ref{curvature})  the index $s$ cannot be simultaneously $>r$ and $\leq r$. From this  we obtain
$$ \sum_{t=1}^n R_{\gamma \bar{\eta}t\bar{t}} = 0, \ \ \ \ \ 
\sum_{t=1}^n R_{i \bar{j}t\bar{t}} =0, \ \ \ \ \ \ \forall \ 1\leq i,j\leq r, \ \ \forall \ r\!+\!1 \leq \gamma, \eta \leq n. $$
By (\ref{mix1}), we have $$4\beta\widehat{R}_{\gamma\bar{\eta}i\bar{j}}=4\beta\widehat{R}_{\gamma\bar{\eta}i\bar{j}}+
\alpha \big(\sum_{t=1}^n R_{\gamma\bar{\eta} t \bar{t}}\delta_{i{j}}+R_{i\bar{j} t\bar{t}}\delta_{\gamma{\eta}}\big)= 2c\big( \delta_{i j} \delta_{\gamma \eta} \big)=0, $$
for any $ 1\leq i,j\leq r$, $r\!+\!1 \leq \gamma, \eta \leq n$. If exactly two of the four indices belong to $\{ r\!+\!1, \ldots , n\}$, then by (\ref{eq:alphaat34}), we get
$$R_{\gamma \bar{\eta}i\bar{j}}=4\widehat{R}_{\gamma \bar{\eta}i\bar{j}}=0. $$
If only one of the four indices belongs to $\{ r\!+\!1, \ldots , n\}$, then it is
$$R_{\gamma \bar{j}k\bar{\ell}}=   \sum_{s=1}^n \big( D^s_{k\gamma } \overline{D^s_{\ell j} } -  D^{\ell}_{s\gamma } \overline{D^k_{sj} }\big).$$
This equals to zero since for each of the two terms on the right-hand side, the second factor is zero by (\ref{eq:goal}). Similarly, when all four indices are in $\{ 1, \ldots , r\}$, $R_{i\bar{j}k\bar{\ell}}=0$ by (\ref{eq:goal}). Therefore, we have shown that $R=0$, completing the proof of Theorem \ref{thm1b}.
 \end{proof}

\vspace{0.3cm}

\section{Almost abelian Lie algebras}
In this section, we focus on Lie-Hermitian manifolds where ${\mathfrak g}$ is an {\em almost abelian Lie algebra}, meaning that ${\mathfrak g}$ is non-abelian but contains an abelian ideal ${\mathfrak a}$ of codimension $1$. As mentioned in the introduction, almost abelian Lie algebras are special cases of Lie algebras containing a $J$-invariant abelian ideal of codimension $2$, so the proof of this section is implied by the proof in the next section. However, we include it here for the convenience of the readers, as the two proofs are analogous while the proof here is more transparent. For notations we will follow  \cite{GuoZ}.

Given an almost abelian Lie algebra \(({\mathfrak g}, J, g)\) with an abelian ideal ${\mathfrak a}\subset{\mathfrak g}$ of codimension 1, write ${\mathfrak a}_J := {\mathfrak a} \cap J {\mathfrak a}$. It is a \(J\)-invariant ideal in ${\mathfrak g}$ of codimension 2. Consequently, we can always choose a unitary basis $\{ e_1, \ldots , e_n\}$ of ${\mathfrak g}^{1,0}$, called an {\em admissible frame}, so that ${\mathfrak a}$ is spanned by
$$ {\mathfrak a} = \text{span}\left\{ \sqrt{-1}(e_1 - \overline{e}_1), \, (e_i + \overline{e}_i), \, \sqrt{-1}(e_i - \overline{e}_i); \, 2 \leq i \leq n \right\}. $$
Let $e$ be an admissible frame of ${\mathfrak g}$. Following \cite{GuoZ}, since ${\mathfrak a} $ is abelian we have
$$ [e_i, e_j]= [ e_i, \overline{e}_j] = [e_1-\overline{e}_1, e_i]= 0, \ \ \ \ \forall \ 2\leq i,j\leq n.$$
From this and equation (\ref{CandD2}) we obtain the following relations
\begin{equation*}
C^{\ast}_{ij} = D^j_{\ast i} = D^1_{\ast i} = C^{\ast}_{1i}+\overline{D^i_{\ast 1}} =0, \ \ \ \ \forall \ 2\leq i,j\leq n.
\end{equation*}
Also, since ${\mathfrak a}$ is an ideal, $[e_1+\overline{e}_1, {\mathfrak a} ] \subseteq {\mathfrak a}$, which leads to $C^1_{1i}=0$ and $D^1_{11} = \overline{D^1_{11}}$. Combining these results, we know that the only possibly non-trivial components of $C$ and $D$ are
\begin{equation} \label{CandD-aala}
D^1_{11}=\lambda \in {\mathbb R}, \ \ \ D^1_{i1}=v_i \in {\mathbb C}, \ \ \ D^j_{i1}=A_{ij}, \ \ \ C^{j}_{1i} = - \overline{A_{ji}}, \ \ \ \ \ \ \forall \  2\leq i,j\leq n.
\end{equation}
Rewriting the structure equation (\ref{eq:structure}) in terms of the coframe $\varphi$ dual to $e$, we obtain
\[
\begin{cases}
d\varphi_1 = - \lambda \,\varphi_1\wedge \overline{\varphi}_1 , \\
d\varphi_i = - \overline{v}_i \, \varphi_1\wedge \overline{\varphi}_1 +  \sum_{j=2}^n \overline{A_{ij}} \,(\varphi_1 + \overline{\varphi}_1)\wedge \varphi_j , \ \ \ \forall \ 2\leq i\leq n.
\end{cases}
\]
For the proof of case (iii) of Theorem \ref{thm}, we will utilize the characterization result from Proposition $7$ and Lemma $6$ in \cite{GuoZ}, which pertains to almost abelian Hermitian Lie algebras.

\begin{lemma} [\cite{GuoZ}] \label{propAL} Let ${\mathfrak g}$ be an almost abelian Lie algebra equipped with a Hermitian structure $(J, g)$, and let $e$ be an admissible frame. Then the following hold:
\begin{itemize}
    \item[(i)] ${\mathfrak g}$ is unimodular $\ \Longleftrightarrow \ $ $\lambda + \mbox{tr}(A) + \overline{\mbox{tr}(A)} = 0$;
    \item[(ii)] $g$ is Chern flat $\iff \lambda = 0$, $v = 0$, $[A, A^*] = 0$.
\end{itemize}
Here $A^*$ stands for the conjugate transpose of the matrix $A$.
\end{lemma}

It was calculated in \cite{GuoZ} that under any admissible frame $e$, the only possibly non-zero components of the Chern torsion and curvature are
\begin{eqnarray}
\begin{cases}
T^1_{1i} = v_i, \ \ \ \ \ \ T^j_{1i} = A_{ij} + \overline{A_{ji}}, \\
R_{1\bar{1}1\bar{1}} = -2\lambda^2 - |v|^2,  \ \ \quad R_{1\bar{1}i\bar{1}} = -\sum_{k=2}^n \overline{A_{ki}}\,v_k, \\
R_{1\bar{1}i\bar{j}} = v_i\overline{v_j} + [A, A^*]_{ij} - \lambda(A_{ij} + \overline{A_{ji}}),
\end{cases} \label{almostcur}
\end{eqnarray}
for any $2 \leq i,j \leq n$, where $[A,A^{\ast}]_{ij} = \sum_{k=2}^n \big( A_{ik}\overline{A_{jk} } - \overline{A_{ki}}A_{kj}\big) $. 

\begin{proof}[{\bf Proof of Theorem \ref{thm} for unimodular almost abelian Lie algebras}]
Let $(\mathfrak g, J, g)$ be a unimodular almost abelian Lie algebra with a Hermitian structure. Assume that the Chern connection of $g$ has constant mixed curvature: $\mathcal{C}_{\alpha, \beta}=c$, where $\beta \neq 0$. Let $e$ be an admissible frame of ${\mathfrak g}$. By (\ref{almostcur}) we have
$R_{i\bar{i}*\bar{*}}= R_{i\bar{i}i\bar{i}} = 0 $ for any $ 2 \leq i \leq n $. So when $\mathcal{C}_{\alpha, \beta}=c$, by (\ref{chun1}) for any  $2\leq i\leq n$, we get $ c = 0 $.

By (\ref{chun3}), (\ref{CandD-aala}) and (\ref{almostcur}), we have $R_{1\bar{1}1\bar{1}} = -2\lambda^2 - |v|^2,$ and
\begin{equation}
 R_{1\bar{1}k\bar{k}} = |v_k|^2-\lambda(A_{kk}+\overline{A_{kk}})+
\sum_{r=2}^n(|A_{kr}|^2-|A_{rk}|^2)), \ \ \ \ \forall \ 2\leq k\leq n. \label{chun5}
\end{equation}
Therefore,
$$ R_{1\bar{1}} = R_{1\bar{1}1\bar{1}} +  \sum_{s=2}^n R_{1\bar{1}s\bar{s}} = -2\lambda^2 - |v|^2 + |v|^2 - \lambda (\mbox{tr} (A) +\overline{\mbox{tr} (A)} ) = -\lambda ^2.$$ 
In the last equality we used the fact that $\mbox{tr} (A) +\overline{\mbox{tr} (A)} = -\lambda$ by part (i) of Lemma \ref{propAL} since ${\mathfrak g}$ is unimodular. By (\ref{chun1}) for $i=1$, we get 
\begin{equation}
0 = \beta R_{1\bar{1}1\bar{1}}+\alpha  R_{1\bar{1}} = - (\alpha+2\beta)\lambda^2-\beta |v|^2. \label{v1}
\end{equation}
Similarly, by (\ref{chun2}) for $i=1$ and any $2\leq k\leq n$ we get
$$ 0 = 4\beta \widehat{R}_{1\bar{1}k\bar{k}} + \alpha (R_{1\bar{1}} + R_{k\bar{k}}) = \beta R_{1\bar{1}k\bar{k}} + \alpha R_{1\bar{1}}.$$ 
Summing $k$ from $2$ to $n$ we deduce
\begin{equation} \label{v2}
0= \beta (|v|^2 + \lambda^2) - (n-1) \alpha \lambda^2 . 
\end{equation}
Adding up (\ref{v1}) and (\ref{v2}), we end up with $(n\alpha+\beta)\lambda^2=0$. If  $(n\alpha+\beta)\neq 0$, then we have  $\lambda=0$, and by either (\ref{v1}) or (\ref{v2}) and the assumption that $\beta \neq 0$, we conclude that $v=0$. If on the other hand $(n\alpha+\beta)= 0$, then $\alpha = -\frac{\beta}{n}$ and (\ref{v1}) gives us 
$$ 0 = -\beta \big(  (2-\frac{1}{n})\lambda^2 + |v|^2  \big). $$
Since $\beta\neq 0$, this gives us $\lambda =0$ and $v=0$.  

In order to prove that $g$ is Chern flat, it remains to show that $R_{1\bar{1}i\bar{j}}=0$ for any $2\leq i,j\leq n$. Since $c=0$, $R_{1\bar{1}}=0$, and $R_{i\bar{\ast}\ast \bar{\ast}}=0$, by  (\ref{mix1}) we have
$$ 0 = 4\beta \widehat{R}_{1\bar{1}i\bar{j}} = \beta R_{1\bar{1}i\bar{j}}. $$
Hence $R_{1\bar{1}i\bar{j}}=0$ as $\beta \neq 0$. This completes the proof of Theorem \ref{thm} in the almost abelian case.
\end{proof}

\vspace{0.3cm}

\section{Lie algebras with $J$-invariant abelian ideal of codimension $2$}
In this section, we focus on Lie algebras that contain $J$-invariant abelian ideals of codimension 2. Let ${\mathfrak g}$ be a Lie algebra which contains an abelian ideal ${\mathfrak a} \subseteq {\mathfrak g}$ of codimension $2$. Then ${\mathfrak g}$ is always solvable of step at most 3, but in general it may not be 2-step solvable. The codimension of ${\mathfrak a}_J := {\mathfrak a} \cap J {\mathfrak a}$ is either $2$ or $4$, and it is $2$ when and only when  $ J{\mathfrak a} = {\mathfrak a} $. Throughout this section, we will focus on this simpler case and assume that $J{\mathfrak a} = {\mathfrak a}$ from now on. We will once again follow the notations of \cite{GuoZ}.

A unitary basis $\{ e_1, \ldots , e_n\}$ of ${\mathfrak g}^{1,0}$ will be called an {\em admissible frame} for $({\mathfrak g},J,g)$ if
$$ {\mathfrak a} = \mbox{span} \{ e_i+\overline{e}_i, \, \sqrt{-1}(e_i-\overline{e}_i); \ 2\leq i\leq n\} . $$
Since ${\mathfrak a}$ is abelian and it is an ideal, we have
\begin{equation*}
C^{\ast}_{ij}=D^j_{\ast i} =C^1_{\ast \ast} = D^{i}_{1\ast} =D^{\ast}_{1i} = 0, \ \ \ \ \ \forall \ 2\leq i,j\leq n,
\end{equation*}
hence the only possibly non-zero  components of $C$ and $D$ are
\begin{equation}  \label{CD-XYZ}
C^{j}_{1i}= X_{ij}, \ \ \ D^1_{11}=\lambda, \ \ \ D^j_{i1}=Y_{ij}, \ \ \ D^1_{ij}=Z_{ij}, \ \ \ D^1_{i1}=v_i, \ \ \ \ \ 2\leq i,j\leq n,
\end{equation}
where $\lambda \geq 0$, $v\in {\mathbb C}^{n-1}$ is a column vector and $X$, $Y$, $Z$ are $(n-1)\times (n-1)$ complex matrices. Here and from now on we have rotated the angle of $e_1$ to assume that $D^1_{11}\geq 0$.  Let $\varphi$ be the coframe dual to $e$. For simplicity, we denote the column vector $^t\!(\varphi_2, \ldots , \varphi_n)$ by $\varphi$, then the structure equation takes the form
\begin{equation}
\left\{ \begin{split} &d\varphi_1 \, = \, -\lambda \varphi_1\overline{\varphi}_1 , \hspace{4.4cm} \\
&d\varphi \, = \, - \varphi_1\overline{\varphi}_1 \overline{v} -\varphi_1\,^t\!X \varphi + \overline{\varphi}_1 \overline{Y} \varphi - \varphi_1 \overline{Z} \,\overline{\varphi}. \end{split} \right.  \label{structure}
\end{equation}
Since $d^2\varphi =0$, we get
\begin{equation}
\left\{ \begin{split} \lambda (X^{\ast}\!+Y)+ [X^{\ast} ,Y]  -  Z\overline{Z} \, = \, 0, \\
\lambda Z - ( Z \,^t\!X + Y Z )\, =\, 0. \hspace{1.2cm} \end{split} \right.   \label{XYZ}
\end{equation}
Also, by (\ref{torsion}), we know that the only possibly non-zero components of the Chern torsion are:
\begin{equation} \label{torsion-XYZ}
T^1_{1i}=v_i, \ \ \ \ T^1_{ij}=Z_{ji}-Z_{ij}, \ \ \ \ T^j_{1i}=Y_{ij}-X_{ij}, \ \ \ \ \ 2\leq i,j\leq n.
\end{equation}
From \cite[Proposition 3]{GuoZ}, we have
\begin{lemma} [\cite{GuoZ}]  \label{lemma6}
Let $({\mathfrak g},J,g)$ be a Lie algebra with Hermitian structure and let ${\mathfrak a}\subseteq {\mathfrak g}$ be a $J$-invariant abelian ideal of codimension $2$. Then under any admissible frame,
\begin{enumerate}
\item ${\mathfrak g}$ is unimodular $\Longleftrightarrow \, \lambda - \mbox{tr}(X)+\mbox{tr}(Y)=0$,
\item $g$ is Chern flat  $\Longleftrightarrow \lambda =0$, $v=0$, $Z=0$, $[Y,Y^{\ast}]=0$, and $[Y,X^{\ast}]=0$.
\end{enumerate}
\end{lemma}

Let $e$ be an admissible frame. By $(\ref{curvature}), \ (\ref{HD}),\,(\ref{chun3}),\,(\ref{RhatD})$, $(\ref{CD-XYZ})$, a straight-forward computation in \cite[formula (16)]{LZ25} leads to the following
\begin{eqnarray}
\hspace{2em}
\begin{cases}
R_{1\bar{1}1\bar{1}} = -2\lambda^2 - |v|^2, \quad {R}_{i\bar{i}i\bar{i}} = |Z_{ii}|^2, \quad \widehat{R}_{i\bar{i}k\bar{k}} = \dfrac{1}{4} |Z_{ik} + Z_{ki}|^2, \quad \forall \ 2 \leq i \neq k \leq n, \\
\widehat{R}_{1\bar{1}i\bar{j}} = \dfrac{1}{4} \big(v v^* + [Y, Y^*] - \lambda (Y^* + Y) -Z\overline{Z}- \,^t\!ZZ^* - \,^t\!Z\overline{Z}\big)_{ij}, \quad \forall \ 2 \leq i,j \leq n, \\
\sum_{i=2}^n\widehat{R}_{1\bar{1}i\bar{i}} = \dfrac{1}{4} \big(|v|^2 - \lambda \text{tr}(Y + Y^*) - 2 \text{tr}(Z\overline{Z}) - |Z|^2\big), 
\end{cases} \hspace{-2em}\label{cur-XYZ1}
\end{eqnarray}
where the last line is obtained by taking trace in the previous line. Similarly, by (\ref{curvature}) and (\ref{CD-XYZ}) we have
\begin{equation} \label{cur-XYZ1b}
R_{1\bar{1}}=\sum_{s=1}^n R_{1\bar{1}s\bar{s}}=-2\lambda^2 - \lambda \, \text{tr}(Y+Y^*), \quad \quad R_{i\bar{j}}= \sum_{s=1}^n R_{i\bar{j}s\bar{s}}=0,\quad  \ \forall \ 2 \leq i,j \leq n.
\end{equation}

Since ${\mathfrak g}$ is unimodular, by part (i) of Lemma \ref{lemma6} we have $\text{tr}(Y-X) = -\lambda$. Also, by taking trace in the first line of (\ref{XYZ}) we obtain
\begin{eqnarray}
\text{tr}(Z\overline{Z}) = \lambda \,\text{tr}(X^* + Y) = \lambda \,\text{tr}(Y^* + Y) + \lambda^2. \label{unimodular-XY}
\end{eqnarray}

\begin{proof}[{\bf Proof of Theorem \ref{thm} for Lie algebras with $J$-invariant abelian ideal of codimension $2$}]
Let $(\mathfrak g, J, g)$ be a unimodular Hermitian Lie algebra which contains an abelian ideal ${\mathfrak a}$ of codimension $2$ with $J{\mathfrak a}={\mathfrak a}$. Assume that the Chern connection of $g$ has constant mixed curvature $\mathcal{C}_{\alpha, \beta}=c$, where $\beta \neq 0$. By (\ref{mix1}), (\ref{chun1}), (\ref{cur-XYZ1}), and (\ref{unimodular-XY}), we have
\begin{align}
c &= \beta R_{1\bar{1}1\bar{1}} + \alpha \sum_{s=1}^n R_{1\bar{1}s\bar{s}} = -(\alpha + 2\beta)\lambda^2 - \beta \,|v|^2 - \alpha \,\text{tr}(Z\overline{Z}), \label{R1} \\
c &= \beta R_{i\bar{i}i\bar{i}} + \alpha \sum_{s=1}^n R_{i\bar{i}s\bar{s}} = \beta \,|Z_{ii}|^2, \quad \forall \ 2 \leq i \leq n, \label{R2} \\
2c &= 4\beta \widehat{R}_{i\bar{i}k\bar{k}} + \alpha  \big( R_{k\bar{k}} + R_{i\bar{i}}  \big) = \beta \,|Z_{ik} + Z_{ki}|^2, \quad \ \forall \ 2 \leq i\neq  k \leq n, \label{R3} \\
2c &= 4\beta \widehat{R}_{1\bar{1}i\bar{i}} + \alpha  \big( R_{1\bar{1}} + R_{i\bar{i}} \big) =  \alpha \big( -\lambda^2 - \text{tr}(Z\overline{Z}) \big) \,+  \label{R4} \\
& \ \ \ +\beta  \big( v v^* + [Y, Y^*] - \lambda (Y^* + Y) - Z\overline{Z} - \,^t\! ZZ^* - \,^t\! Z\overline{Z} \big)_{ii} , \ \ \ \forall \ 2\leq i\leq n.  \nonumber  
\end{align}
Note that (\ref{R2}) and (\ref{R3}) give us
$$ 2c(1+\delta_{ik}) = \beta \, |Z_{ik} + Z_{ki}|^2, \quad \ \forall \ 2 \leq i, k \leq n. $$
Summing $i$ and $k$ from $2$ to $n$, we get
\begin{equation} \label{eqA}
 c \,n(n-1) = \frac{1}{2} \beta \,|\,^t\!Z+Z|^2 = \beta (|Z|^2 +  \mbox{tr}(Z\overline{Z})). 
 \end{equation}
On the other hand, by summing $i$ from $2$ to $n$ in (\ref{R4}) and utilizing (\ref{unimodular-XY}), we get
\begin{equation} \label{eqB}
2c(n - 1) = \beta|v|^2 + (\beta-(n-1)\alpha)\lambda^2 -\beta (|Z|^2 +  \mbox{tr}(Z\overline{Z})) - (2\beta +\alpha (n-1)) \text{tr}(Z\overline{Z}). 
\end{equation}

If we perform a unitary change on $\{ e_2, \ldots , e_n\}$, say $\tilde{e}_i=\sum_{j=2}^n A_{ij}e_j$ for $2\leq i\leq n$,  while keep $\tilde{e}_1=e_1$, then by the definition of the structure constants (\ref{CandD}), we have
$$\widetilde{Z}_{ij} = \widetilde{D}^1_{ij} =\langle \tilde{e}_i, [\overline{\tilde{e}}_1,\tilde{e}_j] \rangle = A_{ik} D^1_{k\ell} A_{j\ell} = (AZ\,^t\!\!A)_{ij}.$$
That is, $\widetilde{Z}=AZ\,^t\!\!A$, hence $(\widetilde{Z}+\,^t\!\widetilde{Z}) = A(Z+\,^t\!Z)\,^t\!\!A$. As it is well-known, for any given symmetric matrix $P$, there always exists a unitary matrix $A$ such that $AP\,^t\!\!A$ is diagonal. So by performing a unitary change of $\{ e_2, \ldots , e_n\}$ if necessary, we may assume that the matrix $Z+\,^t\!Z$ is diagonal. Note that $e$ remains to be admissible. 

Now suppose that $n\geq 3$. Then by (\ref{R3}), we get $c=0$. By (\ref{R2}) and the fact that $\beta\neq 0$, we get $Z_{ii}=0$ for each $2\leq i\leq n$, hence $Z$ is skew-symmetric. This gives us $\mbox{tr}(Z\overline{Z})=-|Z|^2$, while (\ref{R1}) and (\ref{eqB}) become 
\begin{eqnarray}
&& (\alpha +2\beta)\lambda^2 + \beta |v|^2 - \alpha |Z|^2 \ = \ 0,  \label{eqC} \\
&& (\beta - (n-1)\alpha)\lambda^2 + \beta |v|^2 + (2\beta +(n-1) \alpha) |Z|^2 \ = \ 0.  \label{eqD} 
\end{eqnarray}
Let us divide the discussion into the following three cases:

{\em Case 1:}  $\frac{\alpha}{\beta} = - \frac{4}{2n+1}$. \ In this case after we cancel the factor $\beta\neq 0$ the equation (\ref{eqC}) becomes 
$$ 2(2n-1)\lambda^2 + (2n+1)|v|^2 +4|Z|^2 =0. $$
Therefore we obtain $\lambda =0$, $v=0$, and $Z=0$. 

{\em Case 2:}  $\frac{\alpha}{\beta} < - \frac{4}{2n+1}$. \ We can multiply, respectively, (\ref{eqC}) and (\ref{eqD}) by $(n-1)\alpha -\beta$ and $\alpha +2\beta$. Summing the results, we get
$$ (n\frac{\alpha}{\beta}+1)|v|^2 +((2n+1)\frac{\alpha}{\beta} + 4) |Z|^2 =0. $$
The coefficients of both terms on the left hand side are negative, thus $v=0$ and $Z=0$. Plug into (\ref{eqD}), since
$$ 1-(n-1)\frac{\alpha}{\beta} \,>\, 1+(n-1)\frac{4}{2n+1} \ = \ \frac{6n-3}{2n+1} \, >\, 0, $$
we conclude  that $\lambda =0$ as well.

{\em Case 3:}  $\frac{\alpha}{\beta} > - \frac{4}{2n+1}$. \ We can multiply, respectively, (\ref{eqC}) and (\ref{eqD}) by $(n-1)\alpha +2\beta$ and $\alpha $. Summing the results, we get
$$ ((2n+1)\frac{\alpha}{\beta} + 4) \lambda^2 + (n\frac{\alpha}{\beta} +2)|v|^2 =0. $$
In this case the coefficients of both terms on the left  hand side are positive, thus $\lambda =0$ and $v=0$. Plug into (\ref{eqD}) again, since
$$ 2+(n-1)\frac{\alpha}{\beta} \,> \,2-(n-1)\frac{4}{2n+1} \,=\, \frac{6}{2n+1} >0, $$
we conclude  that $Z=0$ as well.

In summary, when $n\geq 3$, we always have $c=0$, $\lambda =0$, $v=0$, and $Z=0$. Plug these into (\ref{R4}) and the first line of (\ref{XYZ}), we obtain $[Y,Y^{\ast}]=0$ and $[X^{\ast}, Y]=0$, respectively. By part (ii) of Lemma \ref{lemma6}, we know that the metric $g$ is Chern flat. 

It remains to discuss the $n=2$ case. In this case $v$, $X$, $Y$, $Z$ are all complex numbers. Since ${\mathfrak g}$ is unimodular, we have $X=\lambda +Y$, and (\ref{XYZ}) now becomes 
\begin{equation} \label{2Da}
 \lambda^2 + \lambda (Y+\overline{Y})=|Z|^2, \ \ \ \ \ \ \ \ YZ=0. 
 \end{equation}
Also, equations (\ref{R1}), (\ref{R2}) and (\ref{eqB}) now tell us that $c=\beta |Z|^2$ and 
\begin{eqnarray} 
 && (2\beta +\alpha )\lambda^2 + \beta |v|^2 +(\beta + \alpha ) |Z|^2 \ = \ 0, \label{2Db} \\
 && (\beta -\alpha )\lambda^2 + \beta |v|^2 -(6\beta +\alpha ) |Z|^2 \ = \ 0. \label{2Dc}
 \end{eqnarray}
The second equation of (\ref{2Da}) says that either $Y=0$ or $Z=0$, so let us again divide our discussion into the following two cases:

{\em Case 1:} If $Z=0$, \ then $c=0$, and equations (\ref{2Db}) and (\ref{2Dc}) gives 
$$ (2\beta +\alpha) \lambda^2 +\beta |v|^2 =0, \ \ \ \  (\beta -\alpha) \lambda^2 +\beta |v|^2 =0. $$
If $\beta +2\alpha \neq 0$, then these two linear equations are independent, thus $\lambda =v=0$. If $\beta +2\alpha = 0$, then $\alpha = -\frac{\beta}{2}$ and the first equation above becomes $\frac{3}{2}\lambda^2+|v|^2=0$ after we cancel the factor $\beta \neq 0$. So again we will have $\lambda =v=0$. By part (ii) of Lemma \ref{lemma6}, we know that the metric $g$ is Chern flat. 

{\em Case 2:} If $Y=0$, \ then we have $\lambda^2=|Z|^2$, so (\ref{2Db}) and (\ref{2Dc}) become
$$ (3\beta +2\alpha) \lambda^2 +\beta |v|^2 =0, \ \ \ \  -(5\beta +2\alpha) \lambda^2 +\beta |v|^2 =0. $$
If $2\beta + \alpha \neq 0$, then the above two equations are linearly independent, thus $\lambda = v=0$, and we get $g$ Chern flat as before. In the following, let us assume that $\alpha = -2\beta$. The above equations only give us $\lambda^2 =|v|^2$ instead. Without loss of generality, let us assume that $\beta =1$. We have $\alpha =-2$, $c=\lambda^2=|v|^2=|Z|^2$, and we are in the ${\mathcal C}_{-2,1}=c$ situation. We claim that $c$ must be zero in this case. To see this, let us apply formula (\ref{mix1}) in Lemma \ref{lemma1} to the case $i=j=k=1$ and $\ell =2$:
$$ 4\beta \widehat{R}_{1\bar{1}1\bar{2}} + \alpha (2R_{1\bar{2}}) = 0 , $$
or equivalently, after plug in $\alpha =-2\beta$ and cancel the factor  $\beta \neq 0$,
\begin{equation} \label{eqE}
 (R_{1\bar{1}1\bar{2}} +  R_{1\bar{2}1\bar{1}}) - 2R_{1\bar{2}} \ = \ R_{1\bar{1}1\bar{2}} - R_{1\bar{2}1\bar{1}} - 2R_{1\bar{2}2\bar{2}} \ = \ 0 . 
 \end{equation}
Now since the only possibly non-zero $D$ components are $D^1_{11}=\lambda$, $D^1_{21}=v$ and $D^1_{22}=Z$, by (\ref{curvature}) we have 
$$ R_{1\bar{1}1\bar{2}} = R_{1\bar{2}1\bar{1}} = -v\overline{Z}, \ \ \  R_{1\bar{2}2\bar{2}} = v\overline{Z}. $$
Plug into (\ref{eqE}), we get $-2v\overline{Z}=0$. Since $c=\lambda^2=|v|^2=|Z|^2$, we conclude that we must have $c=\lambda =v=Z=0$, therefore $g$ is Chern flat. This completes the case (iv) of Theorem \ref{thm}. 
\end{proof}

\vspace{0.3cm}

\section{Bismut torsion-parallel manifolds}
For $n\geq 3$, Bismut torsion-parallel (BTP) manifolds can be balanced (and non-K\"ahler). The class of balanced BTP manifolds is highly restrictive yet intriguing. In the $n=3$ case, a recent study by Zhao-Zheng \cite{ZhaoZ25} provided a classification of all balanced BTP threefolds, allowing us to verify Conjecture \ref{conj2}  on a case-by-case basis. Below,  we will first recall the classification result for balanced BTP threefolds.

Let $(M^3,g)$ be a  balanced, non-K\"ahler compact BTP Hermitian threefold. As shown in \cite{ZhaoZ22}, for any point $p\in M$, there always exists a unitary frame $e$ (henceforth referred to as {\em special frames}) around $p$ such that the only possibly non-zero Chern torsion components are $a_i=T^i_{jk}$, where $(ijk)$ is a cyclic permutation of $(123)$. Furthermore, each $a_i$ is a global constant on $M^3$, with $a_1=\cdots =a_r>0$, $a_{r+1}=\cdots =0$, where $r\in \{ 1,2,3\}$ is the rank of the {\em Streets-Tian tensor} $B$, which is the global $2$-tensor on any Hermitian manifold defined  by
$ B_{i\bar{j}} = \sum_{k,\ell } T^j_{k\ell} \overline{   T^i_{k\ell }  }$ under any unitary frame. Note that the tensor $B$ was first introduced and studied by Streets and Tian in \cite[Formula (4)]{ST11}, where it was denoted as $Q^2$. According to \cite{ZhaoZ22, ZhaoZ25}, any compact balanced (but non-K\"ahler) BTP threefold must be one of the following:
\begin{itemize}
\item $r=3$, $(M^3,g)$ is a compact quotient of the complex simple Lie group $SO(3,{\mathbb C})$, in particular it is Chern flat;

\item  $r=1$, $(M^3,g)$ is the so-called {\em Wallach threefold}, where $M^3$ is biholomorphic to the flag variety ${\mathbb P}(T_{{\mathbb P}^2} )$ while $g$ is the K\"ahler-Einstein metric minus the square of the null-correlation section. 

\item $r=2$, in this case $(M^3,g)$ is said to be of {\em middle type}. With an appropriate scaling of the metric, the Bismut curvature matrix under $e$ becomes
\begin{equation} \label{eq:middletype}
\Theta^b = \left[ \begin{array}{ccc} d\alpha & d\beta_0 & \\ - d\beta_0 & d\alpha & \\ & & 0 \end{array} \right], \ \ \ \ \ \ \ \ \ \left\{ \begin{array}{ll}  d\alpha = \ x(\varphi_{1\bar{1}}+\varphi_{2\bar{2}}) +iy(\varphi_{2\bar{1}}- \varphi_{1\bar{2}}), & \\ d\beta_0 = -iy (\varphi_{1\bar{1}}+\varphi_{2\bar{2}}) + (x\!-\!2)(\varphi_{2\bar{1}}-\varphi_{1\bar{2}}), & \end{array} \right.
\end{equation}
where $x,y$ are real-valued local smooth functions.
\end{itemize}

\begin{proof}[{\bf Proof of Theorem \ref{thm2}}]
Let $(M^3,g)$ be a compact balanced BTP Hermitian threefold. Assume that $g$ is not K\"ahler. Suppose that the Chern connection $\nabla$ of $g$ has constant mixed curvature $c$. Under a special frame $e$, the only possibly non-zero components of Chern torsion are $T^1_{23}=a_1$, $T^2_{31}=a_2$, and $T^3_{12}=a_3$. By (\ref{mix1}), (\ref{eq:Rbhat}) and Lemma \ref{lemma2}, we have
\begin{eqnarray}  
& &  4\beta R^b_{i\bar{j}k\bar{\ell}}+
\alpha \sum_{s=1}^3 \big(R_{i\bar{j}s\bar{s}}\delta_{k{l}}+R_{k\bar{j}s\bar{s}}\delta_{i{l}}+
R_{i\bar{l}s\bar{s}}\delta_{k{j}}+R_{k\bar{l}s\bar{s}}\delta_{i{j}}\big) \label{eq:BTP}\\
&=& \beta \sum_{s=1}^3\big\{\!-2 T^s_{ik} \overline{   T^s_{j\ell}} -3 T^j_{is} \overline{   T^k_{\ell s}} - 3 T^l_{ks} \overline{   T^i_{js } } + T^l_{is} \overline{   T^k_{js } } + T^j_{ks} \overline{   T^i_{\ell s} }\big\} +2c\big( \delta_{i j} \delta_{k \ell} + \delta_{i \ell } \delta_{k j} \big) , \nonumber
\end{eqnarray}
for any $1\leq i,j,k,\ell \leq 3$.
Using the fact that $T^i_{i\ast }=0$ under a special frame, by letting $i=j$ and $k=\ell$ in the above  we obtain
\begin{eqnarray} \label{eq:Hr1}
&& \beta R^b_{i\bar{i}i\bar{i}}+\alpha R_{i\bar{i}}=c,  \label{eq:Hr1} \\
&&  4\beta R^b_{i\bar{i}k\bar{k}}+\alpha \big(R_{i\bar{i}}+R_{k\bar{k}}\big) = 2c+\beta\big(-2a_j^2 +a_k^2+ a_i^2\big), \ \ \ \mbox{if} \ \{ i,j,k\} =\{ 1,2,3\}. \label{eq:Hr1a}
\end{eqnarray}
We will divide the proof into cases according to the rank $r$ of the Streets-Tian tensor $B$. In the case in which $r=3$, $g$ is Chern flat. Thus, the result holds. So we just need to consider the $r=1$ and $r=2$ cases.

\vspace{0.1cm}

{\em Case 1:} When the rank of the Streets-Tian tensor $B$ is $r=1$. \ In this case $(M^3,g)$ is the Wallach threefold described in \cite{ZhaoZ22, ZhaoZ25}. It is a compact homogeneous Hermitian threefold. By (8.15) - (8.19) in \cite{ZhaoZ22} (note that there was a typo in (8.19)), we know that for any given point $p\in M$ there exists a local holomorphic coordinate $\{ z_1, z_2, z_3\}$ such that $p=(0,0,0)$, and at $p$ it holds that  $g_{i\bar{j}}(0)=\delta_{ij}$ and  under the natural frame the components of the Chern curvature tensor $R$  are 
\begin{eqnarray*}
&& R_{i\bar{j}k\bar{\ell}} = 0 ,\ \ \ \ \ \mbox{if} \ \ \{i,k\} \neq \{ j,\ell\}; \\
&& R_{1\bar{1}1\bar{1}} = R_{2\bar{2}2\bar{2}} = R_{3\bar{3}3\bar{3}} = 2, \ \ \ \ \ R_{2\bar{2}1\bar{1}} = R_{2\bar{2}3\bar{3}}  = 1; \\
&& R_{1\bar{2}2\bar{1}} =R_{3\bar{2}2\bar{3}} =1, \ \ \ R_{1\bar{3}3\bar{1}} =-1;  \\
&& R_{1\bar{1}3\bar{3}} = R_{3\bar{3}1\bar{1}}  = R_{1\bar{1}2\bar{2}}  = R_{3\bar{3}2\bar{2}}  =0.
\end{eqnarray*}
From this we see that the at the origin, the components of first Chern Ricci curvature are
$$ R_{1\bar{1}} = R_{3\bar{3}}  = 2, \ \ \ R_{2\bar{2}}=4, \ \ \ R_{i\bar{j}}=0 \ \ \ \mbox{if}  \ i\neq j. $$
For any given type $(1,0)$ tangent vector $X=\sum_{i=1}^3X_i \frac{\partial}{\partial z_i}$ at $p$, we compute
\begin{eqnarray*}
R_{X\overline{X}X\overline{X}} & = & \sum_{i,j,k,\ell }X_i \overline{X}_j  X_k \overline{X}_{\ell} R_{i\bar{j}k\bar{\ell}} \, = \, \sum_i |X_i|^4R_{i\bar{i}i\bar{i}} + \sum_{i\neq k} |X_iX_k|^2 \big(R_{i\bar{i}k\bar{k}} + R_{i\bar{k}k\bar{i}}  \big)  \nonumber \\
& = & 2\sum_i |X_i|^4 + \sum_{i<k} |X_iX_k|^2 \big( R_{i\bar{i}k\bar{k}} + R_{k\bar{k}i\bar{i}} + R_{i\bar{k}k\bar{i}} + R_{k\bar{i}i\bar{k}}  \big) \nonumber \\
& = & 2\sum_i |X_i|^4 +3|X_1X_2|^2 + 3|X_2X_3|^2 - 2 |X_1X_3|^2  \nonumber \\
& = & 2(t^2-2s +|X_2|^4) + 3t|X_2|^2 -2s  \nonumber \\
& = & 2 |X_2|^4 + 3t |X_2|^2 + (2t^2 -6s).  
\end{eqnarray*}
Here for convenience we have denoted by $t=|X_1|^2+|X_3|^2$ and $s=|X_1X_3|^2$. Similarly, 
\begin{eqnarray*}
|X|^2 R_{X\overline{X}} & = & |X|^2 \sum_{i}|X_i|^2 R_{i\bar{i}} \ = \ |X|^2  \big( 2|X_1|^2+4|X_2|^2 + 2|X_3|^2)  \nonumber \\
& = & (t+ |X_2|^2)\, ( 2t+ 4|X_2|^2) \  = \  4 |X_2|^4 + 6t |X_2|^2 + 2t^2 .  \label{valueRic}
\end{eqnarray*}
By the definition of the mixed curvature ${\mathcal C}_{\alpha, \beta}$, at $p$ we have
\begin{eqnarray*}
&& {\mathcal C}_{\alpha ,\beta}(X)-c|X|^4 \ = \  \alpha |X|^2 R_{X\overline{X}} + \beta R_{X\overline{X}X\overline{X}} -c |X|^4 \ = \ \nonumber \\
&& = \ \{ 2(2\alpha +\beta) -c\} |X_2|^4 + \{ 3t(2\alpha +\beta) -2ct\} |X_2|^2 + \{ 2(\alpha +\beta)t^2-6\beta s-ct^2\}.
\end{eqnarray*}
Therefore, if ${\mathcal C}_{\alpha ,\beta}=c$ is a constant, then since $|X_2|^2\geq 0$  is arbitrary, the three coefficients must be identically zero:
$$ 2(2\alpha +\beta)-c=0, \ \ \ \ \ t\{ 3(2\alpha +\beta)-2c\} =0, \ \ \ \ \ t^2\{ 2(\alpha +\beta)-c\} -6\beta s =0. $$
Let $t$ take a positive value, then the first two equalities imply that $c=0$ and $2\alpha +\beta=0$, so the third equality becomes
$$ \beta (t^2-6s) = 0. $$
Take $X_1=1$ and $X_3=0$ for instance, then $t=1$ and $s=0$, so the above gives us $\beta =0$, hence $\alpha =0$. This computation shows the following:

\vspace{0.1cm}

{\em For the Wallach threefold $(M^3,g)$, the mixed curvature ${\mathcal C}_{\alpha, \beta}$ can never be a constant for any $(\alpha ,\beta )\neq (0,0)$. }

\vspace{0.2cm}

{\em Case 2:} When the rank of the Streets-Tian tensor $B$ is $r=2$. \ In this case $(M^3,g)$ is of the middle type, where the Bismut curvature under a special frame is given by (\ref{eq:middletype}). We have
$$ R^b_{\ast \bar{\ast} 3\bar{3}} =0, \ \ \ R^b_{1\bar{1}1\bar{1}}=R^b_{1\bar{1}2\bar{2}}=
R^b_{2\bar{2}1\bar{1}}=R^b_{2\bar{2}2\bar{2}}=x. $$
Note that on any balanced Hermitian manifold, the first Chern Ricci tensor and the first Bismut Ricci tensor coincide, namely,  $R_{i\bar{j}}= R^b_{i\bar{j}}$. In our case we have $\mbox{tr}(\Theta^b)=2\,d\alpha =2 \Theta^b_{11}$, so 
$$ R_{i\bar{j}}= R^b_{i\bar{j}} = 2R^b_{i\bar{j}1\bar{1}}.$$ By (\ref{eq:Hr1}) and (\ref{eq:Hr1a}) we have 
\begin{equation*}
\left\{ \begin{array}{ll} \beta R^b_{1\bar{1}1\bar{1}}+2\alpha R^b_{1\bar{1}1\bar{1}}=c, & \\
\\
\beta R^b_{3\bar{3}3\bar{3}}+2\alpha R^b_{3\bar{3}1\bar{1}}=c, & \\
\\
 4\beta R^b_{2\bar{2}1\bar{1}}+2\alpha \big(R^b_{2\bar{2}1\bar{1}}+R^b_{1\bar{1}1\bar{1}}\big) = 2c+2\beta a_1^2,  & \\
\end{array}\right.
\end{equation*}
The middle line says that $c=0$, while the other two give us
$$ (2\alpha +\beta )x=0, \ \ \ \ \ (2\beta + 2\alpha)x = \beta a_1^2. $$
When $\beta \neq 0$, the second equation above tells us that $x\neq 0$, thus by the first equation we conclude that $2\alpha + \beta =0$. Thus we have shown that

\vspace{0.1cm}

{\em For any balanced BTP threefold of middle type, the mixed curvature ${\mathcal C}_{\alpha ,\beta}$ can never be a non-zero constant. It also cannot be identically zero when $\beta \neq 0$ and $2\alpha +\beta \neq 0$. } 

\vspace{0.1cm}

Note that for $\beta =0$, the mixed curvature is just the first Chern Ricci, given by 
$$ 2d\alpha  = 2x(\varphi_{1\bar{1}}+\varphi_{2\bar{2}}) +2iy(\varphi_{2\bar{1}}- \varphi_{1\bar{2}}).$$
Hence ${\mathcal C}_{1,0}=0$ identically if and only if $x=y=0$  everywhere on $M^3$. In this case $d\alpha =0$ and $d\beta_0=2\varphi_{1\bar{2}} -2\varphi_{2\bar{1}}$. This case can occur. 

Next let us assume that $\beta \neq 0$ but $2\alpha + \beta =0$. Without loss of generality let us assume that $\alpha =1$ and $\beta =-2$. We claim that the mixed curvature ${\mathcal C}_{1,-2}$ cannot be identically zero. Assume otherwise, namely, ${\mathcal C}_{1,-2}\equiv 0$. Then from the equation $(2\beta + 2\alpha)x = \beta a_1^2$ above we conclude that $x=a_1^2>0$. 

On the other hand, from (\ref{eq:middletype}) we get $R^b_{2\bar{1}1\bar{2}}=iy$ where $y$ is a local smooth function which takes real values. Take $i=\ell =2$ and $j=k=1$ in (\ref{eq:BTP}), we get
$$ 4\beta (iy) + \alpha (R_{1\bar{1}} + R_{2\bar{2}}) = \beta \sum_s (-3|T^2_{1s}|^2 -3 |T^1_{2s}|^2). $$
Since $R_{1\bar{1}}=R^b_{1\bar{1}}=2R^b_{1\bar{1}1\bar{1}}=2x$ and similarly $R_{2\bar{2}}=2x$, by plugging $\alpha =1$ and $\beta =-2$ in the above equation we obtain
$$ -8iy + 4x = 12a_1^2. $$
Taking the real part, we get $x=3a_1^2$. Comparing this with the equality $x=a_1^2$ that we obtained before, we get $a_1^2=0$, which is a contradiction.  This shows that the assumption ${\mathcal C}_{1,-2}\equiv 0$ cannot hold for any balanced BTP threefold of middle type, and we have completed the proof of Theorem \ref{thm2}.
\end{proof}

Finally let us consider the class of {\em non-balanced} BTP manifolds. Such manifolds were studied in \cite{ZhaoZ24}. Let $(M^n,g)$ be a compact, non-balanced BTP manifold. Given any  point $p\in M$, there always exists the so-called {\em admissible frame} near $p$, which is a local unitary frame $e$  satisfying  
$$ \eta =\lambda \varphi_n, \ \ \ T^n_{ij }=0, \ \ \ T^j_{in} = \delta_{ij}a_i, \ \ \ \ \ \forall \ 1\leq i,j\leq n,$$
where $\eta$ is Gauduchon's torsion $1$-form defined by $\partial (\omega^{n-1})=-\eta \wedge \omega^{n-1}$ with $\omega$ the K\"ahler form of $g$, $\varphi$ is the coframe dual to $e$, and $T^j_{ik}$ are the Chern torsion components under $e$. Here $\lambda$, $a_1, \ldots , a_n$ are globally defined constants on $M^n$, satisfying
$$ \lambda >0, \ \ \ a_n=0, \ \ \ a_1+ \cdots +  a_{n-1}=\lambda. $$
The readers are referred to Definition 1.6 and Proposition 1.7 of \cite{ZhaoZ24} for more details. As a result, the direction $e_n$ is fixed by the Bismut connection, namely, $\nabla^be_n=0$, which implies that $R^b_{\ast \bar{\ast} \ast \bar{n}}=0$. Now let us prove Theorem \ref{thm3}, which states that any compact, non-balanced BTP manifold cannot have non-zero constant mixed curvature. 

\begin{proof}[{\bf Proof of Theorem \ref{thm3}.}]  Let $(M^n,g)$ be a compact, non-balanced BTP manifold. Assume that the mixed curvature is equal to a non-zero constant: ${\mathcal C}_{\alpha,\beta} =c$, $c\neq 0$. We want to derive at a contradiction.  Fix any $p\in M$, and let $e$ be an admissible frame in a neighborhood of $p$. From our previous discussion, we have $R^b_{\ast \bar{\ast} \ast \bar{n}}=0$. By letting $i=j=k=\ell$ in the formula (\ref{eq:BTP}) at the beginning of this section, we have
$$ 4\beta R^b_{i\bar{i}i\bar{i}} + 4\alpha R_{i\bar{i}} = - 4\beta \sum_{s=1}^n |T^i_{is}|^2 + 4c, \ \ \ \ \ \forall \ 1\leq i\leq n. $$
Apply it to the case $i=n$, and using the property $T^n_{\ast \ast }=0$, we get $\alpha R_{n\bar{n}} = c\neq 0$. 

On the other hand, by taking $i=j=n$ and $k=\ell$ in the first equality in Lemma \ref{lemma2}, we get
\begin{equation}  \label{eq:last}
 R^b_{n\bar{n}k\bar{k}} - R_{n\bar{n}k\bar{k}} = \sum_{s=1}^n |T^k_{ns}|^2 - \sum_{s=1}^n |T^s_{nk}|^2. 
 \end{equation}
Here we used the fact $\nabla^bT=0$ and $T^n_{\ast \ast}=0$ for BTP metrics under admissible frames. Since $R^b_{i\bar{j}k\bar{\ell}}= R^b_{k\bar{\ell}i\bar{j}}$ and $R^b_{i\bar{j}k\bar{n}}=0$, we get $R^b_{n\bar{n}k\bar{k}}=0$. Summing up $k$ from $1$ to $n$ in (\ref{eq:last}), the two terms on the right hand side cancel each other, and we end up with $R_{n\bar{n}}=0$. This of course contradicts to the conclusion $\alpha R_{n\bar{n}} = c\neq 0$ that we obtained before. It shows that for any compact, non-balanced BTP manifold, the mixed curvature ${\mathcal C}_{\alpha , \beta}$ could never be identically equal to a non-zero constant, and Theorem \ref{thm3} is proved.
\end{proof}

\vspace{0.3cm}

\vs

\noindent\textbf{Acknowledgments.} We would like to thank Haojie Chen, Lei Ni, Xiaolan Nie, Kai Tang, Bo Yang, Yashan Zhang, and Quanting Zhao for their interest and helpful discussions. We would also like to thank the referee who made a number of valuable suggestions and corrected dozens of typo and grammar mistakes for us, to which we are very grateful. 

\vs


\begin{thebibliography}{99}

\bibitem {AI} B. Alexandrov and S. Ivanov,  \emph{Vanishing theorems on Hermitian manifolds,} Diff. Geom. Appl. {\bf 14} (2001),  251-265.






\bibitem  {AOUV}  D. Angella, A. Otal, L. Ugarte, and R. Villacampa,  \emph{On Gauduchon connections with K\"ahler-like curvature,} Commun. Anal. Geom. {\bf 30} (2022), no.\,5, 961-1006.




\bibitem  {AT} D. Angella and A. Tomassini, \emph{On the $\partial \overline{\partial}$-lemma and Bott-Chern cohomology,}  Invent. Math. {\bf 192} (2013), 71-81.

\bibitem  {AU}  D. Angella and L. Ugarte,   \emph{Locally conformal Hermitian metrics on complex non-K\"ahler manifolds,} Mediterr. J. Math. {\bf 13} (2016), 2105-2145.




\bibitem {ADM} V. Apostolov, J. Davidov, and O. Muskarov, \emph{Compact self-dual Hermitian surfaces,}  Trans. Amer. Math. Soc. {\bf 348} (1996), 3051-3063.




\bibitem  {AL}  R.M. Arroyo and R.A. Lafuente, \emph{The long-time behavior of the homogeneous pluriclosed flow,} Proc. Lond. Math. Soc.\,(3) {\bf 119} (2019), no.\,1, 266-289.


\bibitem{Belgun} F. Belgun, \textit{On the metric structure of non-K\"ahler complex surfaces}, Math. Ann. \textbf{317} (2000), 1-40.



\bibitem {BG} A. Balas and P. Gauduchon, \emph{Any Hermitian metric of constant nonpositive (Hermitian) holomorphic sectional curvature on a  compact complex surface is K\"ahler,} Math. Zeit. {\bf 190} (1985), 39-43.

\bibitem{BT} K. Broder and K. Tang, {\it On the weighted orthogonal Ricci curvature}, J. Geom. Phys. {\bf 193} (2023), Paper No. 104783, 13 pp.


\bibitem {Bismut} J.-M. Bismut, \emph{A local index theorem for non-K\"ahler manifolds,}  Math. Ann. {\bf 284} (1989), no.\,4, 681-699.




\bibitem {Boothby} W. Boothby, \emph{Hermitian manifolds with zero curvature,} Michigan Math. J. {\bf 5} (1958), no.\,2, 229-233.






\bibitem {CCN} H. Chen, L. Chen, and X. Nie, \emph{Chern-Ricci curvatures, holomorphic sectional curvature and  Hermitian metrics,}  Sci. China Math. {\bf 64} (2021), 763-780.


\bibitem {CN} H. Chen and X. Nie, \emph{Compact Hermitian surfaces with pointwise constant Gauduchon holomorphic sectional curvature,}  Math. Zeit. {\bf 302} (2022), 1721-1737.

\bibitem {ChenZ} S. Chen and F. Zheng, \emph{On Strominger space forms,}  J. Geom. Anal. {\bf 32} (2022), no.\,4, Paper No.\,141, 21pp.


\bibitem {ChenZ1} S. Chen and F. Zheng, \emph{Bismut torsion parallel manifolds with constant holomorphic sectional curvature,} arXiv: 2405.09110.


\bibitem {ChenZ2}S. Chen and F. Zheng, \emph{Canonical metric connections with constant holomorphic sectional curvature,} Pacific J Math, {\bf 334} (2025), no.\,2, 329-348. 

\bibitem{CLT} J. Chu, M. Lee, and L. Tam, {\it K\"{a}hler manifolds and mixed curvature}, Trans. Amer. Math. Soc. {\bf 375} (2022), no.\,11, 7925-7944.

\bibitem{CLZ} J. Chu, M. Lee, and J. Zhu, {\it On K\"{a}hler manifolds with non-negative mixed curvature }, arXiv:2408.14043.

\bibitem{CLT} J. Chu, M. Lee, and L. Tam, {\it K\"{a}hler manifolds and mixed curvature}, Trans. Amer. Math. Soc. {\bf 375} (2022), no.\,11, 7925-7944.





\bibitem{CFGU} L. Cordero, M. Fern\'{a}ndez, A. Gray, and L. Ugarte, \emph{Compact nilmanifolds with nilpotent complex structures: Dolbeault cohomology}, Trans. Amer. Math. Soc. {\bf 352} (2000), no.\,12, 5405-5433.











\bibitem{DGM} J. Davidov, G. Grantcharov, and O. Muskarov, \emph{ Curvature properties of the Chern connection of twistor spaces,} Rocky Mt. J. Math. {\bf 39} (2009), no.\,1, 27-48.









\bibitem{FP3} A. Fino and F. Paradiso, \emph{Hermitian structures on a class of almost nilpotent solvmanifolds,} J. Algebra {\bf 609} (2022), 861-925.









\bibitem {FinoTomassini09} A. Fino and A. Tomassini, \emph{Non K\"ahler solvmanifolds with generalized K\"ahler structure,} J. Symplectic Geom. {\bf 7} (2009), no.\,2, 1–14.















\bibitem {Fu} J-X Fu,  {\em On non-K\"ahler Calabi-Yau threefolds with balanced metrics,} Proceedings of the International Congress of Mathematicians. Volume II, 705-716, Hindustan Book Agency, New Delhi, 2010.









\bibitem {GiustiPodesta} F. Giusti and F. Podest\`a, \emph{Real semisimple Lie groups and balanced metrics,} Rev. Mat. Iberoam. {\bf 39} (2023), no.\,2, 711-729.



\bibitem {GuoZ} Y. Guo and F. Zheng, \emph{Hermitian geometry of Lie algebras with abelian ideals of codimension $2$,} Math. Zeit. {\bf 304} (2023), no.\,3, Paper No 51, 24pp.



\bibitem {GuoZ2} Y. Guo and F. Zheng, \emph{Streets-Tian Conjecture on several special types of Hermitian manifolds,} arXiv:2409. 09425, to appear in Annali di Matematica Pura ed Applicata.



\bibitem {HuangZ} X. Huang and F. Zheng, \emph{On solvmanifolds with complex commutator and constant holomorphic sectional curvature,} Intern. J. Math. {\bf 36} (2025), no.\,6, art no 2550006. arXiv:2501.00810.





\bibitem {LV} J. Lauret and  E.A.R. Valencia, \emph{On the Chern-Ricci flow and its solitons for Lie groups,}  Math. Nachr.
{\bf 288} (2015), no.\,13, 1512-1526.







\bibitem{LZ} Y. Li and F. Zheng, \emph{Complex nilmanifolds with constant holomorphic sectional curvature,}  Proc. Amer. Math. Soc. {\bf 150} (2022), no.\,1,  319-326.

\bibitem{LZ25} Y. Li and F. Zheng, \emph{A note on almost abelian groups with constant holomorphic sectional curvature,}  arXiv: 2503.00415, to appear in Proc. Amer. Math. Soc.

\bibitem{Mat1} S. Matsumura, {\it On the image of MRC fibrations of projective manifolds with semi-positive holomorphic sectional curvature}, Pure Appl. Math. Q. {\bf 16} (2020), no.\,5, 1419-1439.


\bibitem{Mat2} S. Matsumura, {\it On projective manifolds with semi-positive holomorphic sectional curvature}, Amer. J. Math. {\bf144} (2022), no.\,3, 747-777.


\bibitem{N1} L. Ni, {\it Liouville theorems and a Schwarz lemma for holomorphic mappings between K\"{a}hler manifolds}, Comm. Pure Appl. Math. {\bf74} (2021), no.\,5, 1100-1126.

\bibitem{N2} L. Ni, {\it The fundamental group, rational connectedness and the positivity of K\"{a}hler manifolds}, J. Reine Angew. Math. {\bf774} (2021), 267-299.


\bibitem{NZ1} L. Ni and F. Zheng, {\it Comparison and vanishing theorems for K\"{a}hler manifolds}, Calc. Var. Partial Differential Equations {\bf 57} (2018), no.\,6, Paper No. 151, 31 pp.








\bibitem{RZ} P. Rao and F. Zheng, \emph{ Pluriclosed manifolds with constant holomorphic sectional curvature,} Acta. Math. Sinica (English Series). \textbf{38} (2022), no.\,6,  1094-1104.



\bibitem {Salamon} S. Salamon, \emph{Complex structures on nilpotent Lie algebras,} J. Pure Appl. Algebra {\bf 157} (2001), 311-333.











\bibitem{ST11} J. Streets and  G. Tian, \textit{Hermitian curvature flow,} J. Eur. Math. Soc. (JEMS) {\bf 13} (2011), no.\,3, 601–634,


\bibitem {Strominger} A. Strominger, \emph{Superstrings with Torsion,} Nuclear Phys. B {\bf 274} (1986), 253-284.




\bibitem {STW} G. Sz\'ekelyhidi, V. Tosatti and B. Weinkove, \emph {Gauduchon metrics with prescribed volume form,}  Acta Math. {\bf 219} (2017), no.\,1, 181-211.



\bibitem {Tang} K. Tang, \emph{Holomorphic sectional curvature and K\"ahler-like metric} (in Chinese),  Scientia Sinica Mathematica  {\bf 50} (2020), 1-12.

\bibitem{Tang2} K. Tang, {\it Quasi-positive curvature and vanishing theorems}, arXiv:2405.03895.

\bibitem {Tang25} K. Tang, \emph{On mixed curvature for Hermitian manifolds,} arxiv:2501.03749.


\bibitem{Tosatti} V. Tosatti, \emph{Non-K\"ahler Calabi-Yau manifolds,} in Analysis, complex geometry, and mathematical physics: in honor of Duong H. Phong, 261-277, Contemp. Math. 644, Amer. Math. Soc., Providence, RI, 2015.


\bibitem{TosattiYang} V. Tosatti and X. Yang, {\it An extension of a theorem of Wu-Yau,} J. Differential Geom. 107 (2017), no.\,3, 573-579.



\bibitem  {Ugarte} L. Ugarte, \emph{Hermitian structures on six-dimensional nilmanifolds,} Transform. Groups {\bf 12} (2007), 175-202.





\bibitem {VYZ} L. Vezzoni, B. Yang, and F. Zheng, \emph{ Lie groups with flat Gauduchon connections,} Math. Zeit. \textbf{293} (2019), Issue 1-2, 597-608.

\bibitem{WuYau} D. Wu and S. Yau, {\it Negative holomorphic curvature and positive canonical bundle,} Invent. Math. {\bf 204} (2016), no.\,2, 595-604.

\bibitem {WYZ} Q. Wang, B. Yang, and F. Zheng, \emph{On Bismut flat manifolds,} Trans. Amer. Math. Soc. {\bf 373} (2020), 5747-5772.

\bibitem{Yang2018} X. Yang, {\it RC-positivity, rational connectedness and Yau's conjecture}, Camb. J. Math. {\bf 6} (2018), no.\,2, 183-212.

\bibitem {YZ18Cur} B. Yang and F. Zheng, \emph{On curvature tensors of Hermitian manifolds,} Comm. Anal. Geom. {\bf 26} (2018), no.\,5, 1193-1220.

\bibitem {YZ18Gau} B. Yang and F. Zheng, \emph{On compact Hermitian manifolds with flat Gauduchon conmnections,} Acta Math. Sinica (English Series), {\bf 34} (2018), 1259-1268.

\bibitem  {XYangZ} X. Yang and F. Zheng, {\em On real bisectional curvature for Hermitian manifolds,} Trans. Amer. Math. Soc. {\bf 371} (2019),  2703-2718.



\bibitem{Yau} S.-T. Yau,  {\it Seminar on differential geometry}, Ann. of Math Stud. {\bf 102} (1982), 669-706.




\bibitem {ZZ-Crelle} Q. Zhao and F. Zheng, \emph{Strominger connection and pluriclosed metrics,}  J. Reine Angew. Math. (Crelles) {\bf 796} (2023), 245-267.

\bibitem {ZZ-JGP} Q. Zhao and F. Zheng, \emph{Complex nilmanifolds and K\"ahler-like connections,} J. Geom. Phys. {\bf 146} (2019), 103512, 9pp.




\bibitem {ZhaoZ22} Q. Zhao and F. Zheng, \emph{On Hermitian manifolds with Bismut-Strominger parallel torsion,} arXiv:2208.03071.

\bibitem {ZhaoZ24} Q. Zhao and F. Zheng, \emph{Curvature characterization of Hermitian manifolds with Bismut parallel torsion,} arXiv: 2407.14097.

\bibitem {ZhaoZ25} Q. Zhao and F. Zheng, \emph{On balanced Hermitian threefolds with parallel Bismut torsion,} arXiv: 2506.15141.




\bibitem{ZZ} S. Zhang and X. Zhang, {\it Compact K\"{a}hler manifolds with quasi-positive holomorphic sectional curvature}, arXiv:2311.18779.


\bibitem {ZhouZ}  W. Zhou and F. Zheng, \emph{ Hermitian threefolds with vanishing real bisectional curvature} (in Chinese), Scientia Sinica Mathematica {\bf 52} (2022), no.\,7, 757-764. English version: arXiv: 2103.04296.


\end{thebibliography}
\end{document}